\documentclass[10pt]{article}

\makeatletter
\newfont{\footsc}{cmcsc10 at 8truept}
\newfont{\footbf}{cmbx10 at 8truept}
\newfont{\footrm}{cmr10 at 10truept}
\makeatother
\newcommand{\ignore}[1]{}

\newtheorem{definition}{Definition}
\newtheorem{theorem}{Theorem}
\newtheorem{lemma}{Lemma}

\newtheorem{corollary}{Corollary}

\newtheorem{proposition}{Proposition}

\newenvironment{proof}{\begin{trivlist}
                       \item[]{\bf Proof}
                       \hspace{0cm} }{\hfill {\large $\bullet$}
                       \end{trivlist}}

\def\quad{ ~ }
\def\qquad{ ~~ }

\input{epsf}

\def\sssp{\def\baselinestretch{1.2}\large\normalsize}
\def\ssp{\def\baselinestretch{1.25}\large\normalsize}

\def\dsp{\def\baselinestretch{1.6}\large\normalsize}

\title{A Multilinear Operator for Almost Product Evaluation of Hankel Determinants }

\author{ \ \\
{\normalsize \"Omer E\u{g}ecio\u{g}lu }\\
{\normalsize Department of Computer Science, }
{\normalsize University of California,}\\
{\normalsize Santa Barbara CA 93106 \   ({\tt omer@cs.ucsb.edu})}\\ \ \\
{\normalsize Timothy Redmond}\\
{\normalsize Stanford Medical Informatics, Stanford University}\\
{\normalsize Stanford, CA 94305 \ ({\tt tredmond@stanford.edu})}\\ \ \\
{\normalsize Charles Ryavec}\\
{\normalsize College of Creative Studies, University of California,} \\
{\normalsize Santa Barbara CA 93106 \   ({\tt ryavec@ccs.ucsb.edu})} }

\begin{document}
\date{}

\maketitle
\date{}

\begin{abstract}

In a recent paper we have presented a method to evaluate certain 
Hankel determinants as {\em almost products}; i.e.  as a sum of a small number of
products. The technique to find the explicit form of the almost product
relies on differential-convolution equations and trace calculations.
In the trace calculations a number of intermediate nonlinear 
terms involving determinants occur, but  only to
cancel out in the end. 

In this paper, we introduce a class of multilinear  operators $ \gamma$ acting
on tuples of matrices as an alternative to the trace method.
These operators 
do not produce extraneous nonlinear terms, and can be combined
easily with differentiation.

The paper is self contained.
An example of an almost product evaluation using $\gamma$-operators 
is worked out in detail and tables of 
the $\gamma$-operator values on various forms of matrices are provided.
We also present an explicit evaluation of a new class of Hankel determinants and conjectures. 
\\


\end{abstract}

\date{}
%\today

\ssp

\section{Introduction}
\label{introduction}

The expansion of a determinant
$$
\det [a_{i,j}]_{0 \leq i,j \leq n}
$$
from first principles involves calculating the signed sum 
of $(n+1)!$ individual products. This type of an evaluation is not of 
much interest, and one usually uses the multilinearity
of the determinant to obtain more succinct expressions for a given family of determinants. 
Those determinants 
which  may be evaluated as a single product of simple factors (such as the Vandermonde and Cauchy 
determinants) have a special appeal. For product form evaluations, LU decomposition, continued fractions and 
Dodgson condensation are some of the available methods that have been utilized with considerable 
success. There exists an extensive literature on this topic, 
going back to the treatise of Muir 
\cite{Muir90,Muir60}. A more recent compilation of the state of affairs of the theory of determinants 
appears in  Krattenthaler
\cite{K99,K05}, in which 
a wide range of techniques used 
to study the evaluation of families of determinants 
%with product form evaluations 
are described,
accompanied by an extensive
bibliography on the subject.

Of particular interest
are Hankel determinants, for which 
$$
a_{i,j} = a_{i+j} ~.
$$
Certain classes of Hankel determinants with combinatorially interesting entries $ a_{i+j} $ 
have product representations with startling
evaluations, and we mention
$$
\det \left[ { 3 (i+j) +2 \choose i+j} \right]_{0 \leq i,j \leq n} =
 \prod_{i=1}^n\frac{(6i+4)! (2i+1)!}{2 (4 i+2)!(4i+3)!} 
$$
and
$$
\det \left[ { 3 (i+j) \choose i+j} \right]_{0 \leq i,j \leq n} =
 \prod_{i=1}^n\frac{3(3i+1) (6i)!(2i)!}{ (4i)!(4i+1)!} ~
$$
(see \cite{ASM01} and \cite{3kchoosek08}).
The in-between case of the binomial coefficients
\begin{equation}
\label{bc}
a_k = { 3 k +1 \choose k}
\end{equation}
is not amenable to
standard methods since it does not have a product evaluation.
In a recent paper \cite{ERR07} we proved that 
for the entries (\ref{bc}),
the evaluation is an {\em almost product}; in this case
a sum of $n+1$ products of simple factors:
$$
\det \left[ { 3 (i+j) +1 \choose i+j} \right]_{0 \leq i,j \leq n} =
 \prod_{i=1}^n\frac{(6i+4)! (2i+1)!}{2 (4 i+2)!(4i+3)!}
       \sum_{i=0}^n \frac{ n!(4n+3)!!  (3n+i+2)! }
                              {(3n+2)! i!(n-i)!(4n+2i+3)!!} ~.
$$

The technique presented in \cite{ERR07} to find the explicit form of the almost product
for this particular Hankel determinant
relies on the following steps:
\begin{enumerate}
\item[(I)] Using $k = i+j$, replace $a_k$ with polynomials
\begin{equation}
\label{akintro31}
a_k(x) =  \sum_{m=0}^{k} {3k+1- m \choose k-m} x^m
\end{equation}
so that $ a_k(x)$ is a monic polynomial
of degree $k$ with $ a_k = a_k(0)$. Consequently
the associated Hankel determinant
$H_n(x)$ is a polynomial, and $H_n = H_n(0)$.
\item[(II)] Establish a second order ODE satisfied by $H_n(x)$.
\item[(III)] Solve the DE in (II), and evaluate the solution at $x=0$.
\end{enumerate}

The
{\em $ ( \beta , \alpha)$-case}  of this problem is the evaluation of the Hankel determinants  where 
the entries are
\begin{equation}
\label{ab}
a_k^{(\beta , \alpha )} (x) = \sum_{m=0}^k {\beta k+\alpha -m \choose k-m} x^m ~.
\end{equation}

The bulk of the work is contained in Step (II), and this part of the argument itself relies on 
three essential identities.
%A First Identity (FI), a Second Identity (SI), and a Third Identity (TI).
These identities are linked in the derivation of the differential equation via the application of 
a trace operator.

In this paper, we introduce a class of  multilinear  $ \gamma$-operators acting 
on tuples of matrices which take the place of this trace operator. 

If it had just been a matter of calculating the differential equation in the $(3,1)$-case as 
we did in \cite{ERR07}, then which technique we used 
might not have mattered much. 
However, we wanted to try to extend the differential equation method to a larger class of 
$ ( \beta , \alpha)$-cases, and we found that already
in the $(2,2)$-case, the $\gamma$-operators 
simplified the calculations significantly.
To be specific, 
in the trace approach some nonlinear terms occur in the calculations, which get 
canceled in the end. For example the following ratio of determinants (using the notation in \cite{ERR07})
\begin{equation}
\label{nonl}
- 4 (4n +3)^2 \frac{K_n^2}{H_n}
\end{equation}
appears during the course of the trace calculations (e.g. \cite{ERR07}, p. 15), and is later cancelled.
%The transient nonlinear terms such as (\ref{nonl}) that arise in 
%the calculations of the $(3,1)$-case,

As one goes to other cases, these nonlinear terms proliferate.
In the $(2,2)$-case, there are over half a dozen of these terms that arise, which all 
cancel. 

These nonlinear terms turn out to be an avoidable burden in a method that already involves 
a lot of calculation.  It is easier to combine differentiation with the 
$\gamma$-operators than with the trace calculations of \cite{ERR07} and in addition the $\gamma$-operator 
calculations do not produce the extraneous nonlinear terms mentioned above. 
An added benefit is that they  need not be calculated from 
scratch for other Hankel determinant evaluations.
In Appendix III, we provide extensive table of values of $\gamma$-operators. \\ 

Let
\begin{equation}
\label{aks}
a_k (x) = \sum_{m=0}^k {2k+2-m \choose k-m} x^m
\end{equation}
and define the $(n+1) \times (n+1) $ Hankel determinants by
\begin{equation}
\label{Hn}
    H_n (x) = \det [a_{i+j}(x)]_{0\le i,j\le n} ~.
\end{equation}
A few of these polynomials and the Hankel determinants are  as follows:
\begin{eqnarray*}
a_0(x)&=& 1\\
a_1(x)&=& 4 + x\\
a_2(x)&=& 15 + 5x + x^2\\
a_3(x)&=& 56 + 21x + 6x^2 + x^3
%a_4(x)&=& 210 + 84x + 28x^2 + 7x^3 + x^4
\end{eqnarray*}
and
\begin{eqnarray*}
H_0(x)&=& 1\\
H_1(x)&=&-1 - 3x\\
H_2(x)&=&-1 - x + 5x^2\\
H_3(x)&=& 1 + 6x + 3x^2 - 7x^3 ~.
%H_4(x)&=&1 + 2x - 15x^2 - 5x^3 + 9x^4
\end{eqnarray*}
We give the elements of the application of $\gamma$-operators by working through the proof of the following 
theorem.
\begin{theorem}
\label{thm1}
Suppose $a_k$ and the $H_n(x)$ are as defined in 
(\ref{aks}) and (\ref{Hn}). Then $H_n(x)$ has the following almost product evaluations:
\begin{equation}
\label{thm11}
H_n(x) = (-1)^n \sum_{k=0}^n \left[  (2n+3) {n+k \choose 2k+1} + (2k+1) {n+k+1 \choose 2k+1} \right] (x-2)^k 
%H_n(x) = (-1)^n \sum_{k=0}^n \frac{2 n^2 +4 n+2k^2+1}{2k+1} {n+k \choose 2k} (x-2)^k
\end{equation}
and
\begin{equation}
\label{thm12}
H_n(x) =  \sum_{k=0}^n (-1)^k \left[  ( n+k+1) {n+k \choose 2k} + (2n+4k+1) {n+k \choose 2k+1}+
8(k+1)  {n+k +1 \choose 2k+3} \right] (x+2)^k~.
%H_n(x) = (2n+3) \sum_{k=0}^n (-1)^k \frac{2n^2 + 2k^2 + 4k + 1}{(2k + 1)(2k + 3)} {n + k \choose 2k} (x+2)^k ~.
\end{equation}
\end{theorem}
Alternate expressions for (\ref{thm11}) and (\ref{thm12}) are given in
(\ref{Hnatm2})
and (\ref{Hnat2}).
The expansion of $H_n(x)$ around $ x=0$ can be found in (\ref{solutionat0}).
The generating function of the $ H_n(x)$ itself is given in (\ref{gf}).

It is known that \cite{ERR07,Gessel07,K07}
\begin{equation}
\label{m1}
\det \left[  2 (i +j) +2 \choose i+j \right]_{0 \leq i,j \leq n}  = (-1)^\frac{n(n+1)}{2} ~.
\end{equation}
Our purpose is not the derivation of this relatively simple numerical evaluation itself, but 
to give an exposition of the salient points of the $ \gamma$-operators, which 
allow us to evaluate the general case of the
Hankel determinants of the 
polynomials (\ref{aks}) as an almost product. 

Additionally, we obtain numerical evaluations of 
$H_n(x)$ at special values of $x$. A number of these are presented in Section
\ref{special} and at the end of Section \ref{DEsolution}. 

In Corollary \ref{additional}
we evaluate the Hankel determinant
$$
\det \left[ { 2 (i+j) +3 \choose i+j} \right]_{0 \leq i,j \leq n} ~.
$$

The explicit almost product evaluation of 
Theorem \ref{thm1} is derived from the second order differential equation satisfied by 
these Hankel determinants. 
This differential equation is given in Theorem \ref{detheorem} in Section \ref{derivatives}. 
With the definition of the polynomials
in (\ref{ab}), the evaluation in this paper is the $(\beta, \alpha) = (2,2)$-case.

The outline of the rest of this paper is as follows:
In Section 
\ref{preliminaries}, we define determinants $H_\lambda$ for partitions $\lambda$ obtained from 
a given Hankel matrix. This is followed by the introduction of the  
family of multilinear operators $ \gamma$ along with
their basic properties and a combinatorial 
interpretation for their evaluation in Section \ref{gamma}.
Section \ref{evaluations} presents example calculations 
with the $\gamma$'s, and a compilation of 
evaluations that are used in the paper. 
This is followed by three identities that are typically 
needed for our methods, 
and the derivation of the equations satisfied 
by the various $H_\lambda$ that arise
in the calculations. We obtain a system of first order differential equations 
which results in a second order differential equation for the Hankel determinant 
we wish to evaluate in Section \ref{derivatives}. 
Evaluation at special points are discussed in
Section \ref{special}, and the general solution of the differential equation is derived in 
Section \ref{DEsolution}. An additional Hankel determinant evaluation is given at the end of this 
section in Corollary
\ref{additional}.
In Section \ref{zeros}, we consider the properties of the zeros of the Hankel determinants and show that 
they form a Sturm sequence.
Conjectures on the evaluation of similar Hankel determinants are presented in Section \ref{discussion}.
This is followed by Appendix I - III  where we give the proofs of the 
results stated and used in the calculations as well as tables of $\gamma$-operator evaluations.
We remark that 
Sections 2--4 and Appendices I and III apply to general Hankel 
matrices, whereas Sections 5--10 and Appendix II apply to the evaluation of the case $ \alpha = \beta =
2$.

\section{Preliminaries}
\label{preliminaries}

We consider general Hankel matrices $ A = [ a_{i+j}]_{0 \leq i,j \leq n}$ in the symbols $ a_k$.
 In \cite{ERR07} and in Section \ref{introduction} of the present paper
we used the notation $H_n$ for $\det(A)$.
However, it is useful to have alternate notation for various 
determinants that arise, in which sometimes the
parameter $n$ is suppressed. 
Unless otherwise indicated, we assume
that $n$ has been chosen and is fixed.

A {\em partition} $\lambda $ of an integer $m>0$ is a
weakly decreasing sequence of integers $\lambda =(\lambda
_{1}\geq \lambda _{2}\geq \cdots \geq \lambda_p >0)$ with $m=\lambda
_{1}+\lambda_{2}+\cdots +\lambda_p$. Each
$\lambda_{i}$ is called a \emph{part} of $\lambda $. For
example $\lambda =(3,2,2)$ is a partition of $m=7$ into $p=3$ parts. 

We use the notation $ \lambda = m^{\alpha_m} \cdots 2^{\alpha_2} 1 ^{\alpha_1}$ for a 
partition $\lambda$ of $m$, indicating  that $ \lambda $ has $ \alpha_i$ parts of size $i$. 
Thus for example, $ \lambda = 3^2 2 1^3$ denotes the partition
$3 + 3+ 2 + 1 + 1+1$ of 11. 
Given $n>0$,
each partition  $(\lambda _{1}\geq \lambda_{2}\geq \cdots \geq \lambda_{p}>0)$ with $ p \leq n+1$ 
defines a determinant
of a matrix obtained from the $ (n+1) \times (n+1) $ Hankel matrix
$ A_n = [ a_{i+j}]_{0 \leq i,j \leq n}$
in the symbols $a_k$,
by shifting the column indices of the entries up according to $ \lambda$ as follows:
Let $ \mu_i = \lambda_i$ for $ i =1, \ldots, p$ and $ \mu_i = 0$ for $ i = p+1, \ldots, n+1 $.
Then
$$
H_\lambda  =
\det [ a_{i+j + \mu_{n+1- j}}]_{0 \leq i,j \leq n} ~.
$$
We use the special notation $0$ to denote the sequence $ \mu_i = 0$ for $ i= 1, \ldots, n+1$.
For example when $ n =3$,
$$
   H_0  = \det \left[\begin{array}{cccc}
       a_0     & a_1     &  a_{2}  & a_{3} \\
       a_1     & a_2     & a_{3}    & a_{4} \\
       a_{2} & a_3  & a_{4} & a_5 \\
       a_{3} & a_4 & a_5   & a_6
              \end{array}\right], ~
   H_2 = \det \left[\begin{array}{cccc}
       a_0     & a_1     &  a_{2}  & a_{5} \\
       a_1     & a_2     & a_{3}    & a_{6} \\
       a_{2} & a_3  & a_{4} & a_7 \\
       a_{3} & a_4 & a_5   & a_8
              \end{array}\right], ~
   H_{31^2} = \det \left[\begin{array}{cccc}
       a_0     & a_2     &  a_{3}  & a_{6} \\
       a_1     & a_3     & a_{4}    & a_{7} \\
       a_{2} & a_4  & a_{5} & a_{8} \\
       a_{3} & a_5 & a_6   & a_{10}
              \end{array}\right] ~.
$$
We note that 
these determinants are obtained in a way similar to the expansion of Schur functions 
in terms of the homogeneous symmetric functions by the Jacobi-Trudi identity \cite{M95}.

When the $a_k = a_k(x)$ are functions of $x$, then 
$ H_\lambda = H_\lambda(x)$ is a function of $x$.
When we need to indicate the dependence of the determinant on $n$ as well as $x$, we write 
$$
H_\lambda (n, x)
$$
for the $(n+1) \times (n+1) $ shifted Hankel determinant.
As an example, with this  notation 
(\ref{thm11}) is written as
\begin{equation}
\label{thm13}
H_0(x)=H_0 (n, x)  = (-1)^n \sum_{k=0}^n \frac{2 n^2 +4 n+2k^2+1}{2k+1} {n+k \choose 2k} (x-2)^k ~.
\end{equation} 
The $(n+1) \times (n+1) $ Hankel determinant 
will be denoted by a number of different notations in this paper. Among these are 
$H_n= H_n (x)$, $ H_0 = H_0 (x)$, and $H_0 (n,x)$. In the latter two 
cases it should be clear from the context
that the subscript $0$ refers to the partition involved 
and not to the dimension of the Hankel matrix.

Our aim is to obtain a first order linear system of equations
\begin{eqnarray}
\label{star}
Q \frac{d }{dx}H_0 &=& Q_0 H_0 + Q_1 H_1 \\ \nonumber
U \frac{d }{dx} H_1 &=& U_0 H_0 + U_1 H_1
\end{eqnarray}
where
the coefficients are polynomial functions of $x$ and $n$.
From this system 
the 
second order differential equation for $H_0$ 
given in Theorem \ref{detheorem}
can be found immediately.

In the process of 
differentiating $H_0$ and $H_1$
the following five determinants
$$
H_3,  H_{21}, H_{1^3}, H_2, H_{1^2}
$$
are encountered. We will express each of these in terms of 
the two determinants $ H_0,  H_1 $.

The $ \gamma$-operator that we next define allows us 
to do this from the three identities satisfied by the $a_k$, while avoiding 
having to deal with nonlinear expressions involving 
determinants. This operator has the additional advantage of simplifying 
differentiation of determinants, improving on the trace calculations used in \cite{ERR07}.

\section{The $ \gamma$-operator}
\label{gamma}

We define a multilinear operator $\gamma$ on $m$-tuples of matrices as follows:
\begin{definition}
\label{def1}
Given $(n+1) \times (n+1)$ matrices $A$ and $ X_1, X_2, \ldots , X_m$ with $ m \geq 1$,
define
$$
\gamma_A (~) = \det (A)
$$
and
$$
\gamma_A ( X_1, \ldots, X_m) = 
\partial_{t_1}
\partial_{t_2} \cdots \partial_{t_m}
\det ( A + t_1 X_1 + t_2 X_2 + \cdots + t_m X_m ) |_{t_1 = \cdots = t_m = 0}
$$
where $ t_1, t_2, \ldots, t_m$ are variables that do not appear in 
$A$ or $ X_1, X_2, \ldots , X_m$.
\end{definition}

\ignore
{
\noindent
{\bf Note}: The $ \gamma$'s relate to the more familiar trace formulas as follows:
\begin{eqnarray} 
\label{gx}
\gamma_A (X) & = & \det(A) \mbox{Tr} (A^{-1} X), \\ 
\label{gxy}
\gamma_A (X,Y) &=& \det(A) \Big ( \mbox{Tr} (A^{-1} X)  \mbox{Tr} (A^{-1} Y)  - \mbox{Tr} (A^{-1} X A^{-1} Y)  \Big)  \\
\nonumber
\gamma_A (X,Y,Z) &=& \det(A) \Big( \mbox{Tr} (A^{-1} X)  \mbox{Tr} (A^{-1} Y) \mbox{Tr} (A^{-1} Z)    \\ \nonumber
& & - \mbox{Tr} (A^{-1} X)  \mbox{Tr} (A^{-1} Y (A^{-1} Z)   \\ 
\label{gxyz}
& & -    \mbox{Tr} (A^{-1} Y)  \mbox{Tr} (A^{-1} X (A^{-1} Z) \\ \nonumber
& &    -    \mbox{Tr} (A^{-1} Z)  \mbox{Tr} (A^{-1} X (A^{-1} Y)  \\ \nonumber
& &  + 2 \mbox{Tr} (A^{-1} XA^{-1} Y A^{-1} Z) \Big) ~.   \nonumber
\end{eqnarray}

To prove (\ref{gxy}), for example, 
let
$$
\Theta = A + t_1 X + t_2 Y  ~.
$$
Then 
\begin{eqnarray*}
\partial_{t_1}  \Theta^{-1} & =& - \Theta^{-1} X \Theta^{-1} \\
\partial_{t_1} \det (\Theta) & =& \det ( \Theta ) \mbox{Tr} (\Theta^{-1} X)
\end{eqnarray*}
and similarly for $t_2$. We have
\begin{eqnarray*}
\partial_{t_1}  \partial_{t_2}  \det (\Theta) & =&  
\partial_{t_1} \Big( \det ( \Theta ) \mbox{Tr} (\Theta^{-1} Y) \Big) \\
 & =&  \det( \Theta ) \Big( \mbox{Tr} (\Theta^{-1} X)  \mbox{Tr} (\Theta^{-1} Y)  
- \mbox{Tr} (\Theta^{-1} X \Theta^{-1} Y)  \Big)   ~.
\end{eqnarray*}
Putting $ t_1 = t_2 = 0$ gives  (\ref{gxy}). 

It turns out that the equations (\ref{gx}), (\ref{gxy}), and (\ref{gxyz}) illustrate 
the reason why the $\gamma$-operator is preferable to 
using traces. For example, the right-hand side of 
(\ref{gxy}) is a difference of two types of trace terms. 
In calculating this difference using traces as
introduced in \cite{ERR07}, several identical determinantal expressions  arise from each term 
that get canceled in the end. 
The use of the $\gamma$'s bypasses this redundant calculation.
}
%end ignore

Next we give a computationally feasible combinatorial interpretation of 
$ \gamma_A ( X_1, \ldots, X_m) $ for small $m$, based on elementary properties of determinants.

\begin{definition}
\label{Ass}
Suppose
$A$ and $ X_1, \ldots , X_m$ are
$(n+1) \times (n+1)$ matrices, $m \leq n+1$. 
Given a subset of column indices $S =\{ j_1, j_2 , \ldots, j_m \} \subseteq \{ 0, 1, \ldots, n\}$ 
and a permutation $ \sigma$ of
$ \{1, 2, \ldots, m \}$, $ A_{S, \sigma}$ is defined as the matrix which is obtained from 
$A$ by replacing $A$'s $j_k$-th column by the $ j_k$-th column of the matrix $X_{\sigma_k}$ for 
$ k =1, 2, \ldots , m$.
\end{definition}

With this notation we have

\begin{proposition}
\label{greplace}
For $ m \leq n+1$, 
\begin{equation}
\label{gr}
\gamma_A ( X_1, \ldots , X_m ) = \sum_{S, \sigma} \det ( A_{S, \sigma})
\end{equation}
where the summation is over all subsets 
$S$ of 
$\{ 0, 1, \ldots, n\}$ with $ |S| = m$ and all permutation $ \sigma$ of
$\{ 1, 2, \ldots, m \}$.
\end{proposition}

\noindent
{\bf Note:} 
The expansion (\ref{gr}) is also valid as a sum over row indices where the replacements made are 
rows from $X_1, \ldots,
X_m$ instead of columns.\\

Another motivation for using the $\gamma$-operators is that they differentiate nicely; the derivative of
a $ \gamma$ is a sum of $\gamma$s.
\begin{proposition}
\label{gder}
For $ m \leq n$,
$$
\frac{d}{dx} \gamma_A ( X_1, \ldots, X_m) = 
\gamma_A ( \frac{d}{dx} A , X_1, \ldots , X_m)
+ \sum_{j=1}^m \gamma_{A} ( X_1, \ldots, X_{j-1}, \frac{d}{dx} X_j, X_{j+1}, \ldots , X_m) ~.
$$
\end{proposition}

The proofs of Proposition \ref{greplace} and Proposition \ref{gder}  can be found in Appendix I. \\

Using Proposition \ref{greplace}, we can evaluate $ \gamma_A$ on matrices that are associated with
$A$ in terms of determinants $H_\lambda$ for various partitions $\lambda$.
Next, we give a few examples of these calculations and a compilation of the expansions needed.

\section{Explicit $\gamma_A$ evaluations}
\label{evaluations}

Let 
$$ A = [a_{i+j}]_{0\le i,j\le n} 
$$
We start with a few sample calculations. \\

\noindent
{\bf Example: } 
In the calculation of
$\gamma_A ( [a_{i+j}])$,
the sum in (\ref{gr}) is over all subsets $ S \subseteq \{0, 1, \ldots n \}$ with a single element and $\sigma$ 
is the identity permutation. We 
are replacing a column of $A$ with the same column, so the resulting determinant 
is $H_0= \det (A)$ for each one of $n+1$ 
possible column  selections. Thus
$$
\gamma_A ( [a_{i+j}]) = (n+1) H_0 ~.
$$

\noindent
{\bf Example: } 
In the calculation of
$\gamma_A ( [a_{i+j+2}])$
the sum in (\ref{gr}) is again over all subsets $ S \subseteq \{0, 1, \ldots n \}$ with one element. If $S=\{j\}$ and 
$ j \leq n-2$, then the $j$-th and the $ (j+2)$-nd columns 
are identical in $ A_{S, \sigma}$  and the determinant
vanishes. For $ j=n$, the determinant is $H_2$ and for $ j =n-1$  it is $ - H_{1^2}$. Therefore
$$
\gamma_A ( [a_{i+j+2}]) = H_2 - H_{1^2} ~.
$$

\noindent
{\bf Example: } 
We split the calculation of 
$\gamma_A ( [(i+j) a_{i+j+2}])$ into two pieces:
$$
\gamma_A ( [(i+j) a_{i+j+2}]) = \gamma_A ( [i a_{i+j+2}])+  \gamma_A ( [j a_{i+j+2}]).
$$
In the calculation of 
$\gamma_A ( [j a_{i+j+2}])$, the determinant in (\ref{gr}) survives 
only for $S=\{n\}$ and $ S= \{n-1\}$, exactly as in the
case of the evaluation of $\gamma_A ( [a_{i+j+2}])$ above. However, now the 
determinant gets multiplied by the factor $n$ of the new $n$-th column 
in the former case, and by the factor $n-1$ of the $(n-1)$-st column in the latter. 
Therefore
$$
\gamma_A ( [j a_{i+j+2}]) = n H_2 - (n-1)H_{1^2}
$$
$\gamma_A ( [i a_{i+j+2}])$ evaluates to the same expression, since now we are dealing with rows instead of columns, but
otherwise the argument is the same. Therefore
$$
\gamma_A ( [(i+j) a_{i+j+2}]) = 2n H_2 - 2(n-1)H_{1^2} ~.
$$

\begin{definition}
\label{dcon}
For a polynomial sequence
$ a_n = a_n (x)$ ($ n \geq 0$), the convolution polynomials
$c_n = c_n(x)$ are defined by
$$
c_n = \sum_{k=0}^n a_k a_{n-k}
$$
with $ c_{-1} = 0$.
\end{definition}

\noindent
{\bf Example: } 
To compute
$\gamma_A ( [c_{i+j+1}]) $ for $n=2$, we use the expansion of 
the matrix $[c_{i+j+1}] $ in terms of 
shifted versions of $A$ as given below. The expansion for arbitrary $n$ can be 
found in Appendix I.
\begin{eqnarray} \nonumber
[ c_{i+j+1}]  & =&   
a_0 \left[ 
\begin{array}{ccc}
 a_1 & a_2 & a_3 \\
 a_2 & a_3 & a_4\\
 a_3 & a_4 & a_5 
\end{array} 
\right]
+ a_1 \left[  
\begin{array}{ccc} 
a_0 & a_1 &a_2  \\
a_1 & a_2 &a_3  \\
a_2 & a_3  &a_4
\end{array} 
\right]
+
a_2 \left[
\begin{array}{ccc}
0& a_0& a_1   \\
0& a_1& a_2   \\
0& a_2& a_3   
\end{array}
\right]
+
a_3 \left[
\begin{array}{ccc}
0& 0& a_0   \\
0& 0& a_1   \\
0& 0& a_2   
\end{array}
\right]\\
\label{k1}
&  &
\hspace*{1cm} + ~a_0 \left[
\begin{array}{ccc}
0 & 0 & 0 \\ 
a_2 & a_3 & a_4 \\
a_3 & a_4 & a_5 
\end{array} 
\right]
+
a_1 \left[ 
\begin{array}{ccc}
0& 0 & 0  \\
0& 0 & 0  \\
a_2& a_3 & a_4
\end{array}
\right]
\end{eqnarray}
A routine calculation gives
\begin{eqnarray*}
\gamma_A ( [c_{i+j+1}]_{0 \leq i,j \leq n}) & = & a_0 H_1 +  n a_1 H_0 + a_0 H_1 +  n a_1 H_0 \\
 & = & 2 a_0 H_1 +  2 n a_1 H_0  ~.
\end{eqnarray*}

We provide another example of a $\gamma$ calculation.

\noindent
{\bf Example: } 
In the calculation of
$\gamma_A ( [a_{i+j+1}], [a_{i+j+2}])$
the sum in (\ref{gr}) is over all subsets $ S \subseteq \{0, 1, \ldots n \}$ with two elements. If  $S=\{j_1 < j_2\}$
with $ j_2 \leq n-2$, then for $\sigma = (1)(2)$, the columns 
$j_2$ and $ j_2+2$, and for $ \sigma = (12)$, the columns $j_2$ and $j_2+1$ of  
$ A_{S, \sigma}$ are identical. Therefore in these cases the determinant
vanishes. The remaining possibilities for $   S , \sigma $ pairs can be enumerated as 
\begin{enumerate}
\item $S= \{n-1,n\} \mbox{~and~}  \sigma= (1)(2) $, 
\item $S= \{n-2,n-1\} \mbox{~and~} \sigma= (1)(2)$ ,
\item $S= \{n-2,n\} \mbox{~and~} \sigma= (12)~$ .
\end{enumerate}
The resulting determinants are 
$$
H_{21}, ~ - H_{1^3}, ~ - H_{1^3},
$$
respectively. Therefore
$$
\gamma_A ( [a_{i+j+1}], [a_{i+j+2}]) = 
H_{21} - 2 H_{1^3} ~.
$$

In Tables \ref{table1}-\ref{table4} of Appendix III, 
we give a list of various $ \gamma$ evaluations. 
The ones that are needed for the computations
in this paper are in Tables \ref{table1} and \ref{table2}.

\ignore
{
so that
\begin{eqnarray*}
a_0(x)&=& 1\\
a_1(x)&=& 4 + x\\
a_2(x)&=& 15 + 5x + x^2\\
a_3(x)&=& 56 + 21x + 6x^2 + x^3\\
a_4(x)&=& 210 + 84x + 28x^2 + 7x^3 + x^4\\
a_5(x)&=& 792 + 330x + 120x^2 + 36x^3 + 8x^4 + x^5
\end{eqnarray*}
Define the $(n+1) \times (n+1) $ hankel determinants 
$$
    H_n (x) = \det [a_{i+j}(x)]_{0\le i,j\le n}
$$

First few of these determinants are as follows:

\begin{eqnarray*}
H_0(x)&=& 1\\
H_1(x)&=&-1 - 3x\\
H_2(x)&=&-1 - x + 5x^2\\
H_3(x)&=& 1 + 6x + 3x^2 - 7x^3\\
H_4(x)&=&1 + 2x - 15x^2 - 5x^3 + 9x^4\\
H_5(x)&=& -1 - 9x - 9x^2 + 28x^3 + 7x^4 - 11x^5
%H_6(x)&=&-1 - 3x + 30x^2 + 20x^3 - 45x^4 - 9x^5 + 13x^6
\end{eqnarray*}
% end ignore
}

\section{The three identities}
\label{identities}

Now we consider the $(2,2)$-case. 
The three identities used in the argument are given in the following three lemmas.
These identities are typical of our methods. 
The first identity is a 
differential-convolution equation.
The second identity involves convolutions and $a_k$
but no derivatives. The third identity is a linear dependence among certain column vectors involving the 
$a_k$.
 
\begin{lemma}
\label{FIlemma}
{(First Identity (FI))}  
\begin{eqnarray}\nonumber
(x-2) x (x+2)(3x+2)\frac{d }{dx} a_n & = & 2 n (x-1) a_{n+2}  + ( n (x-6) (x-2)+ 3 x^2-2 x+4)a_{n+1}\\
\label{FI}
&& - ( 3 x^3+18 x^2-20 x+24+4 n (x^2+4)) a_n \\\nonumber
& & +8  (x-1)^2 c_n  -32  (x-1)^2 c_{n-1} 
\end{eqnarray}
\end{lemma}
%We denote the right-hand side of this by $FI(n)$.

\begin{lemma}
\label{SIlemma}
{(Second Identity (SI))}  
\begin{eqnarray} \nonumber
& & (n x + 3 x + 2) a_{n+2} - (n x (x + 6) + 3 x^2 + 16 x + 8) a_{n+1} \\
\label{SI}
 & & + 2x (x + 2)(2n + 5 ) a_n   +(x - 1) (x - 2) c_n -  4 (x - 1) (x - 2) c_{n-1}  =0
\end{eqnarray}
\end{lemma}

\begin{lemma}
{(Third Identity (TI))}  
\label{TIlemma}
\begin{equation}
\label{TI}
\sum_{j=0}^{n+2} w_{n,j}(x) a_{i+j}(x) =0
\end{equation}
for $ i = 0 , 1, \ldots , n $ where
\begin{eqnarray} \nonumber
 w_{n,j}(x) & = & (-1)^{n-j} \left\{ \frac{2(2n+5)}{2j+1} { n+j+2  \choose 2j } 
+ \frac{(2n+3)(2n+5)}{2j+1} { n+j+2  \choose 2j }  x \right. \\
& & \hspace*{2cm} \left . + \frac{(2n+3)(2n+5)}{2j+3} {  n+j+2 \choose 2j+1} x^2 \right\} ~.
\label{TIweights}
\end{eqnarray}
\end{lemma}

The proofs can be found in Appendix II. We remark that the weights in Lemma \ref{TIlemma}
are typical of our method. Once the coefficients of the weight polynomials $ w_{n,j}(x)$ are
guessed, then automatic binomial identity provers can be used to prove (\ref{TI}). \\

\ignore
{
\section{The approach}

We will find coefficients that are polynomial functions of $x$ and $n$ so that

\begin{eqnarray}
Q \frac{d }{dx}H_0 &=& Q_0 H_0 + Q_1 H_1 \\
R \frac{d^2 }{dx^2} H_0 &=& R_0 H_0 + R_1 H_1 
\end{eqnarray}
Then the second order differential equation for $H_0$ can be found by eliminating $H_1$ as
$$
Q_1 R \frac{d^2 }{dx^2} H_0 - R_1 Q \frac{d }{dx}H_0  + (R_1 Q_0- Q_1 R_0 ) H_0=0  
$$
%end ignore
}

To prove (\ref{star}), we will find the expansions of both $ \frac{d }{dx}H_0 $ 
and $ \frac{d }{dx} H_1$ in terms of 
$H_0$ and $H_1$. Since at first other determinants $H_\lambda$ also appear in 
these derivatives, they will need to be eliminated. We do this by 
constructing a sufficient number of equations involving them, and 
then expressing each one in terms of 
$H_0$ and $H_1$.

\section{The five equations}
\label{equations}

\subsection{Equation from $\gamma_A([SI(i+j)])$ }
Apply
$$
\gamma_{A} ( * )
$$ 
to the $(n+1) \times (n+1)$ matrix whose $(i,j)$-th entry is obtained from the second identity
(\ref{SI}) evaluated at $i+j$ and expand using linearity.
If we denote the matrix so obtained from the second identity 
by $[SI(i+j)]$, then
the computation is the 
expansion of
$\gamma_{A} ([SI(i+j)] )=0$.
We obtain
\begin{eqnarray*}
0 & = & 
x \gamma_{A} ([(i+j)a_{i+j+2}])  + (3 x+2) \gamma_{A} ([a_{i+j+2}] )\\
& & -x (x+6) \gamma_{A} ([(i+j)a_{i+j+1}] )  - (3 x^2 + 16 x + 8) \gamma_{A} ( [a_{i+j+1}])\\
& & + 4 x (x+2) \gamma_{A} ( [(i+j)a_{i+j}]) + 10 x (x+2) \gamma_{A} ( [a_{i+j}]) \\
& & + (x-1)(x-2) \gamma_{A} ( [c_{i+j}]) - 4 (x - 1) (x - 2) \gamma_{A} ([c_{i+j-1}])
\end{eqnarray*}
Making use of the
entries in the
$\gamma_A (*)$ computations
from Table \ref{table1}, we get
\begin{eqnarray*}
0 & = & x ( 2n H_2 -2(n-1)H_{1^2}) + (3 x+2) (H_2-  H_{1^2}) \\
& & -x (x+6)  2nH_1  - (3 x^2 + 16 x + 8) H_1 \\
& & + 4 x (x+2) n(n+1) H_0  + 10 x (x+2) (n+1) H_ 0  \\
& & + (x-1)(x-2)  (2 n+1)  H_0  ~.
\end{eqnarray*}

Therefore
\begin{eqnarray}\nonumber
& & (2 + 3 x + 2 n x)H_2 - (2 + x + 2 n x)H_{1^2} - (8 + 16 x + 12 n x + 3 x^2 + 2 n x^2)H_1 \\
\label{equation1}
& & + (2 + 4 n + 17 x + 22 n x + 8 n^2 x + 11 x^2 + 16 n x^2 + 4 n^2 x^2)H_0 = 0 ~.
\end{eqnarray}

\subsection{Equation from  $\gamma_A([SI(i+j+1)])$}

Now apply $\gamma$ to the matrix obtained by evaluating the 
second identity (\ref{SI}) at $i+j+1$. If we denote this
matrix by $[SI(i+j+1)]$, then
this computation is the  expansion of
$\gamma_{A} ([SI(i+j+1)] )=0$ from (\ref{SI}).

\begin{eqnarray*}
0 & = & 
x \gamma_{A} ( [(i+j)a_{i+j+3}])  + (4x+2) \gamma_{A} ([a_{i+j+3}] ) \\
&& -x(x+6) \gamma_{A} ([(i+j)a_{i+j+2}] )   - (4 x^2 + 22 x + 8) \gamma_{A} ( [a_{i+j+2}]) \\
&& +4x(x+2)  \gamma_{A} ([(i+j)a_{i+j+1}]) + 14x(x+2) \gamma_{A} ([a_{i+j+1}]) \\
& & + (x - 1) (x - 2)  \gamma_{A} ( [c_{i+j+1}])  -  4 (x - 1) (x - 2) \gamma_{A} ([c_{i+j}]) ~.
\end{eqnarray*}

Using Table \ref{table1},

\begin{eqnarray*}
0 & = & 
x  (2n H_3 -2(n-1)H_{21}+2(n-2)H_{1^3}) + (4x+2) (H_3 - H_{21} +  H_{1^3})  \\
&& -x(x+6) (2n H_2 -2(n-1)H_{1^2})    - (4 x^2 + 22 x + 8) ( H_2-  H_{1^2})  \\
&& +4x(x+2)  2n H_1 + 14x(x+2) H_1  \\
& & + (x - 1) (x - 2)  (2 H_1 + 2 n (x+4)  H_0)   -  4 (x - 1) (x - 2)(2 n+1) H_0  ~.
\end{eqnarray*}

Therefore
\begin{eqnarray}\nonumber
& &  (1 + 2 x + n x)H_3 - (1 + x + n x)H_{21} + (1 + n x)H_{1^3} \\
\label{equation2}
& & - (4 + 11 x + 6 n x + 2 x^2 + n x^2)H_2 + (4 + 5 x + 6 n x + x^2 + n x^2)H_{1^2} \\ \nonumber
& &  + (2 + 11 x + 8 n x + 8 x^2 + 4 n x^2)H_1 + (-2 + x) (-1 + x) (-2 + n x)H_0 = 0 ~.
\end{eqnarray}

\subsection{Equation from $\gamma_{A} ([a_{i+j+1}],[SI(i+j)] )$}

Now consider the expansion of 
$\gamma_{A} ([a_{i+j+1}],[SI(i+j)] )=0$ from (\ref{SI}).
\begin{eqnarray*}
0 & = & x \gamma_{A} ([a_{i+j+1}], [(i+j)a_{i+j+2}])  + (3x+2) \gamma_{A} ([a_{i+j+1}],[a_{i+j+2}] ) \\
& & -x(x+6) \gamma_{A} ([a_{i+j+1}],[(i+j)a_{i+j+1}] )  - (3 x^2 + 16 x + 8) \gamma_{A} ([a_{i+j+1}],  [a_{i+j+1}]) \\
& & + 4x(x+2) \gamma_{A} ([a_{i+j+1}],  [(i+j)a_{i+j}])  + 10x(x+2) \gamma_{A} ([a_{i+j+1}], [a_{i+j}]) \\
& & + (x-1)(x-2) \gamma_{A} ([a_{i+j+1}],  [c_{i+j}])  -  4 (x - 1) (x - 2) \gamma_{A} ([a_{i+j+1}],
[c_{i+j-1}]) ~.
\end{eqnarray*}
Using the 
$\gamma_A ( [a_{i+j+1}], *)$ computations
from Table \ref{table2}, we get
\begin{eqnarray*}
0 & = & x  ( 2n H_{21} -2(2n-3)H_{1^3})  + (3x+2) ( H_{21}- 2 H_{1^3})  \\
& & -x(x+6)2(2n-1)H_{1^2}   - (3 x^2 + 16 x + 8)   2 H_{1^2} \\
& & + 4x(x+2)  n(n-1) H_1  +  10x(x+2)  n H_1 \\
& & + (x-1)(x-2) ( (2 n-1) H_1- (2n-1) (x+4) H_0)  -  4 (x - 1) (x - 2)  (-2 n H_0).
\end{eqnarray*}
Therefore for $ n \geq 2 $, 
\begin{eqnarray} 
\label{equation3}
 & & (2 + 3 x + 2 n x)  H_{21}  - 4 (1 + n x) H_{1^3}  - 4 (4 + 5 x + 6 n x + x^2 + n x^2) H_{1^2} \\ \nonumber
& & + ( 4 n -2 + 3 x + 6 n x + 8 n^2 x - x^2 + 8 n x^2 + 4n^2 x^2) H_1  
- (x-2 ) (x-1) (2 n x -4 - x ) H_0 =0.
\end{eqnarray}

\subsection{Two equations from the third identity}

The third identity is as given in Lemma \ref{TIlemma}.
Define the column vector
$$
v_j = \left[a_j , a_{j+1} , \ldots ,  a_{j+n} \right]^T ~.
$$
The third identity (\ref{TI}) says that 
the vectors $v_0, v_1, \ldots , v_{n+2}$ are linearly dependent with the 
weights given in (\ref{TIweights}), i.e.
\begin{equation}
\label{zero}
\sum_{j=0}^{n+2} w_{n,j} v_j = 0 ~.
\end{equation}
Now consider the 
determinant of the $ (n+1) \times (n+1)$  matrix whose 
first $n$ columns are the columns of $A$, 
and whose last column is the zero vector. 
Writing the zero vector in the form (\ref{zero}) and expanding the 
determinant by linearity, we find
$$
w_{n,n+2} H_2 + w_{n,n+1} H_1  + w_{n,n} H_0 =0 ~.
$$
Substituting the weights from 
(\ref{TIweights}), this gives the equation
\begin{eqnarray}\nonumber
& & 
(2 + 3x + 2n x)H_2 - (10 + 4n + 15x + 16n x + 4n^2x + 3x^2 + 2n x^2)H_1 \\
\label{equation4}
& & + (n + 1)(2n + 5)(2 + 3x + 2n x + 2x^2)H_0 =0 ~.
\end{eqnarray}

Next we apply the same expansion 
trick to the matrix whose first $n-1$ columns are those of $A$, i.e. 
$ v_0, v_1, \ldots , v_{n-2}$;  whose $(n-1)$-st column is $v_n$; and whose last column is the zero vector, 
written 
in the form (\ref{zero}). Expanding the
determinant by linearity,
this time we obtain
$$
w_{n,n+2} H_{21} + w_{n,n+1} H_{1^2} 
- w_{n, n-1} H_0 = 0 ~.
$$
Therefore another equation is 
\begin{eqnarray} 
\label{equation5}
& & 3(2 + 3x + 2n x)H_{21} - 3 \Big(10 + 4n + 15x + 16n x + 4n^2x + 3x^2 + 2n x^2 \Big)H_{1^2}\\ \nonumber
& & + \Big(2 n (1 + 2 n) (5 + 2 n) + n (1 + 2 n) (3 + 2 n) (5 + 2 n)x + 3n (3 + 2n) (5 + 2 n)x^2 \Big)H_0
= 0 ~.
\end{eqnarray}

\ignore
{
This second equation is a consequence of the general relation
\begin{equation}
\label{idenk}
w_{n,n+2} H_{21^j} + w_{n,n+1} H_{1^{j+1}} - w_{n,n-j} H_0= 0
\end{equation}
which holds for $ j=0, 1, \ldots , n$. This can be seen by computing the determinant of the matrix obtained from
$ A = [ a_{i+j}]_{0 \leq i,j \leq n}$ by replacing column $n-j$ by column $n$, and column $n$ by zeros. 
We then express the zero last column as a sum of column vectors as indicated by the third identity. Expanding, 
all but three determinants vanish, giving (\ref{idenk}).

Note: (\ref{idenk}) is the identity we use to guess third identities in general: For instance with offset 2, first 
guess $w_{n,n+2} , w_{n,n+1} , w_{n,n}$ by linear algebra, then use (\ref{idenk}) to solve for 
$w_{n,n-j} $ and find the pattern for the coefficients by interpolation.
%end ignore
}

Equations 
(\ref{equation1}), 
(\ref{equation2}), 
(\ref{equation3}), 
(\ref{equation4}), 
(\ref{equation5}), 
form  a $ 5 \times 5$ linear system  $M u= b$ which expresses
the determinants
$$
u = \Big[H_3,  H_{21}, H_{1^3}, H_2, H_{1^2}\Big]^{T}
$$
in terms of the two determinants $ H_0,  H_1 $.
The matrix $M$ is as follows:
{\small
$$
\left[
\begin{array}{lllll}
 0 & 0 & 0 & 2 n x+3 x+2 & -2 n x-x-2 \\
 n x+2 x+1 & -n x-x-1 & n x+1 & -n x^2-2 x^2-6 n x-11 x-4 & n
   x^2+x^2+6 n x+5 x+4 \\
 0 & 2 n x+3 x+2 & -4 (n x+1) & 0 & -4 \left(n x^2+x^2+6 n x+5
   x+4\right) \\
 0 & 0 & 0 & 2 n x+3 x+2 & 0 \\
 0 & 3 (2 n x+3 x+2) & 0 & 0 & \begin{array}{l} -3 (4 x n^2+2 x^2 n+16 x n\\ ~~~~~ +4 n+3 x^2+15 x+10) \end{array}
\end{array}
\right]
$$
}
with
$$
\det(M)= 
12 (1 + n x) (1 + 2 x + n x) (2 + x + 2 n x) (2 + 3 x + 2 n x)^2 ~.
$$
Solving $M u= b$ for $u$, we obtain 
each of $H_3,  H_{21}, H_{1^3}, H_2, H_{1^2} $ in terms of $H_0$ and $H_1$.
\begin{eqnarray} \nonumber
3 (2 + 3 x + 2 n x) H_3 &=&  
-2 (n+1) \Big(8 x n^3+12 x^2 n^2+64 x n^2+8 n^2+6 x^3 n\\ \nonumber
\label{expansionH3}
& & +66 x^2 n+162 x n+52 n+15 x^3+90 x^2+126 x+84\Big) H_0 \\ 
& & + 3 \Big(4 x n^3+4 x^2 n^2+32 x n^2+4 n^2+2 x^3 n+18 x^2 n\\ \nonumber
& & +81 x n+26 n+3 x^3+18 x^2+63 x+42\Big) H_1 ~,\\ \nonumber
3 (2 + x + 2 n x) (2 + 3 x + 2 n x) H_{21} &=&
\Big(-64 x^2 n^5-48 x^3 n^4-416 x^2 n^4-128 x n^4-192 x^3 n^3\\
\label{expansionH21}
& & -1040 x^2 n^3-704 x n^3-64 n^3+12 x^4 n^2-192 x^3 n^2\\ \nonumber
& & -1192 x^2 n^2-1360 x n^2-288 n^2+24 x^4 n+24 x^3 n-480 x^2 n\\ \nonumber
& & -1120 x n-416 n+9 x^4+63 x^3+48 x^2-300 x-240 \Big) H_0 \\ \nonumber
& & + 3 (4 x n^2+4 x n+4 n-x+2) \Big(4 x n^2+2 x^2 n +16 x n \\ \nonumber
& & +4 n+3 x^2+15 x+10\Big)H_1~, \\ \nonumber
3 (2 + x + 2 n x) H_{1^3} &=&
(-16 x n^4-32 x n^3-16 n^3+28 x n^2-24 n^2-6 x^3 n-12 x^2 n\\
\label{expansionH111}
& & +80 x n+16 n-3 x^3-12 x^2+12 x+48)H_0 \\ \nonumber
& & + 3 \left(4 x n^3+4 n^2+2 x^2 n-9 x n-2 n+x^2-4\right) H_1 ~,\\ \nonumber
(2 + 3 x + 2 n x) H_2 &=&
-(n+1) (2 n+5) \left(2 x^2+2 n x+3 x+2\right) H_0 \\
\label{expansionH2}
& & + ( 4 x n^2+2 x^2 n+16 x n+4 n+3 x^2+15 x+10)H_1 ~,\\ \nonumber
( 2 + x + 2 n x) H_{1^2} &=&
(-4 x n^3-12 x n^2-4 n^2+2 x^2 n-9 x n-10 n+x^2\\
\label{expansionH11}
& & +2 x-8)H_0  +(4 x n^2+4 x n+4 n-x+2) H_1 ~.
\end{eqnarray}

Equipped with these expansions, we now proceed with the calculation of the derivatives of 
$H_0$ and $H_1$.

\section{The derivatives of $H_0$ and $H_1$}
\label{derivatives}

\subsection{The derivative of $H_0$}

From Definition \ref{def1},
$$
 H_0 = \gamma_A ( ~ )  ~.
$$
Therefore 
by Proposition \ref{gder} we have
$$
\frac{d }{dx} H_0 = \gamma_A ( [ \frac{d }{dx} a_{i+j} ] ) ~.
$$
Using $FI(i+j)$,
\begin{eqnarray*}
(x-2) x (x+2) (3 x+2) \gamma_A ([FI(i+j)]) &= & 
2 (x-1) \gamma_A ([(i+j) a_{i+j+2}])\\
& & +(x-6) (x-2)\gamma_A ([(i+j) a_{i+j+1}] ) \\
& & +(3 x^2-2 x+4) \gamma_A ([a_{i+j+1}] ) \\
& & -4(x^2+4) \gamma_A ([(i+j) a_{i+j}]) \\
&& -(3 x^3+18 x^2-20 x+24) \gamma_A ( [a_{i+j}]) \\
& & +8(x-1)^2 \gamma_A ([c_{i+j} ])\\ 
& & -32(x-1)^2 \gamma_A ([c_{i+j-1} ]) ~.
\end{eqnarray*}
The 
values for $ \gamma_A ( *)$ from Table \ref{table1} give
\begin{eqnarray*}
(x-2) x (x+2) (3 x+2) \gamma_A ([FI(i+j)]) &= & 
2 (x-1) (2n H_2 -2(n-1)H_{1^2}) \\
& & +(x-6) (x-2)2nH_1  \\
& & +(3 x^2-2 x+4) H_1 \\
& & -4(x^2+4)  n(n+1) H_0  \\
&& -(3 x^3+18 x^2-20 x+24)  (n+1) H_ 0  \\
& & +8(x-1)^2 (2 n+1)  H_0   ~.
\end{eqnarray*}
Now using the expressions in (\ref{expansionH2}) and (\ref{expansionH11}) 
for $H_2$ and  $H_{1^2} $ in terms of $ H_1, H_0$, we obtain 
 $\frac{d }{dx} H_0$ as
\begin{equation}
 \label{derivativeH0}
Q \frac{d }{dx} H_0 = Q_0 H_0 + Q_1 H_1
\end{equation}
where
\begin{eqnarray}\nonumber
Q &=& 
(x-2)(x+2) (2 n x+x+2) (2 n x+3 x+2) ~,\\
\label{star1}
Q_0 &=& 
-(n+1) \Big(16 x^2 n^3+4 x^3 n^2+48 x^2 n^2+32 x n^2+8 x^3 n\\\nonumber
&&   +36 x^2 n+80 x n+16 n+3 x^3+12 x^2+12 x+48\Big) ~,\\\nonumber
Q_1 &=& (2 n+3) \Big(4 n^2 x^2+4 n x^2+x^2+8 n x+4\Big) ~.
\end{eqnarray}

\subsection{The derivative of $H_1$}

To differentiate 
$ H_1$ we use the expression
$$
H_1 = \gamma_A ( [a_{i+j+1}])
$$
from Table \ref{table1}.  From Proposition \ref{gder}  we have

$$
\frac{d }{dx} H_1 = \gamma_A (  [a_{i+j+1}] , [\frac{d }{dx}a_{i+j}]) + \gamma_A ( [\frac{d }{dx}
a_{i+j+1}])   ~.
$$
Therefore, to compute $\frac{d }{dx} H_1$ 
$$ \gamma_A ([a_{i+j+1}], [FI(i+j)]) \mbox{~~and~~} 
 \gamma_A ([FI(i+j+1)])
$$ 
are needed.  For the first one of these 
\begin{eqnarray*}
(x-2) x (x+2) (3 x+2) \gamma_A ([a_{i+j+1}], [FI(i+j)]) &= & 
2 (x-1) \gamma_A ([a_{i+j+1}], [(i+j) a_{i+j+2}])\\
& & +(x-6) (x-2)\gamma_A ([a_{i+j+1}], [(i+j) a_{i+j+1}] ) \\
& & +(3 x^2-2 x+4) \gamma_A ([a_{i+j+1}], [a_{i+j+1}] ) \\
& & -4(x^2+4) \gamma_A ([a_{i+j+1}], [(i+j) a_{i+j}]) \\
&& -(3 x^3+18 x^2-20 x+24) \gamma_A ([a_{i+j+1}], [a_{i+j}]) \\
& & +8(x-1)^2 \gamma_A ([a_{i+j+1}], [c_{i+j} ])\\ 
& & -32(x-1)^2 \gamma_A ([a_{i+j+1}], [c_{i+j-1} ]) ~.
\end{eqnarray*}

Using the entries in  
Table \ref{table2} for the $\gamma_A ( [a_{i+j+1}], *)$ computations, we get
\begin{eqnarray*}
(x-2) x (x+2) (3 x+2) \gamma_A ([a_{i+j+1}], [FI(i+j)]) &= & 
2 (x-1) ( 2n H_{21} -2(2n-3)H_{1^3}) \\
& & +(x-6) (x-2)  2(2n-1)H_{1^2}  \\
& & +(3 x^2-2 x+4) 2 H_{1^2}  \\
& & -4(x^2+4) n(n-1) H_1 \\
&& -(3 x^3+18 x^2-20 x+24) n H_1  \\
& & +8(x-1)^2 ( (2 n-1) H_1- (2n-1) (x+4) H_0 )\\ 
& & -32(x-1)^2  (-2 n H_0 ) ~.
\end{eqnarray*}

For the term $\gamma_A ([FI(i+j+1)])$, we obtain
\begin{eqnarray*}
(x-2) x (x+2) (3 x+2) \gamma_A ([FI(i+j+1)]) &= &
2 (x-1) \gamma_A ([(i+j) a_{i+j+3}])\\
& & +2 (x-1) \gamma_A ([ a_{i+j+3}])\\
& & +(x-6) (x-2)\gamma_A ( [(i+j) a_{i+j+2}] ) \\
& & + 2 (8 - 5 x + 2 x^2)\gamma_A ([a_{i+j+2}] ) \\
& & -4(x^2+4) \gamma_A ( [(i+j) a_{i+j+1}]) \\
& & -(40 - 20 x + 22 x^2 + 3 x^3) \gamma_A ([a_{i+j+1}]) \\
& & +8(x-1)^2 \gamma_A ([c_{i+j+1} ])\\
& & -32(x-1)^2 \gamma_A ([c_{i+j} ]) ~.
\end{eqnarray*}

Using Table \ref{table1} this gives
\begin{eqnarray*}
(x-2) x (x+2) (3 x+2) \gamma_A ([FI(i+j+1)]) &= &
2 (x-1) (2n H_3 -2(n-1)H_{21}+2(n-2)H_{1^3}) \\
& & +2 (x-1) ( H_3 - H_{21} +  H_{1^3})\\
& & +(x-6) (x-2)(  2n H_2 -2(n-1)H_{1^2} ) \\
& & + 2 (8 - 5 x + 2 x^2) ( H_2-  H_{1^2})  \\
& & -4(x^2+4)  2n H_1   \\
& & -(40 - 20 x + 22 x^2 + 3 x^3) H_1  \\
& & +8(x-1)^2 (2 H_1 + 2 n (x+4) H_0) \\
& & -32(x-1)^2   (2 n+1)H_0 ~.
\end{eqnarray*}

Adding, we get 
$$
(x-2) x (x+2) (3 x+2)\frac{d }{dx} H_1
$$ 
as a combination of 
$H_3, H_{21}, H_{1^3}, H_2, H_{1^2} , H_1, H_0$. 
After that,
we use the expressions (\ref{expansionH3})-(\ref{expansionH11}) 
for $H_3, H_{21}, H_{1^3}, H_2$ and $H_{1^2}$ and express 
 $\frac{d }{dx} H_1$ as a linear combination of $H_0, H_1$ as
\begin{equation}
 \label{derivativeH1}
U \frac{d }{dx} H_1 = U_0 H_0 + U_1 H_1  ~.
\end{equation}
We find
\begin{eqnarray} \nonumber 
U &=& 
(x-2) (x+2) (2 n x+x+2) (2 n x+3 x+2) ~,\\ \nonumber
U_0 &= & 
-2 (n+1) \Big(16 x^2 n^4+8 x^3 n^3+72 x^2 n^3+32 x n^3\\ \nonumber
& & +28 x^3 n^2+116 x^2 n^2+112 x n^2+16 n^2+26 x^3 n+86 x^2 n\\ \nonumber
& & +104 x n+56 n+7 x^3+22 x^2+20 x+56\Big) ~,\\
\label{star2}
U_1 &= & 
 \Big(16 x^2 n^4+4 x^3 n^3+64 x^2 n^3+32 x n^3+12 x^3 n^2\\ \nonumber
& & +92 x^2 n^2+80 x n^2+16 n^2+11 x^3 n+56 x^2 n+44 x n+32 n\\ \nonumber
& & +3 x^3+10 x^2-4 x+24\Big) ~.
\end{eqnarray}

The explicit polynomials in
(\ref{star1}) and (\ref{star2}) are 
the coefficients
of the system of differential equations (\ref{star}).

Differentiating both sides of (\ref{derivativeH0}) 
%for the derivative $\frac{d }{dx} H_0$ 
and substituting 
the expansions of $\frac{d }{dx} H_0$ and $\frac{d }{dx} H_1$ in terms of $H_0$ and $H_1$, we obtain 
\begin{equation}
\label{derivative2H0}
R \frac{d^2 }{dx^2} H_0 = R_0 H_0 + R_1 H_1  ~,
\end{equation}
where
\begin{eqnarray}\nonumber
R &=& 
(x-2)^2 (x+2)^2 (2 n x+x+2) (2 n x+3 x+2) ~,\\ \nonumber
R_0 &= & 
(n+1) \Big(4 n^3 x^4+16 n^2 x^4+19 n x^4+6 x^4+32 n^3 x^3+96 n^2 x^3\\ \nonumber
& & +64 n x^3+18 x^3-48 n^3 x^2+240 n x^2+48 x^2+128 n^3 x+288 n^2 x\\
\label{Rs}
& & +160 n x+264 x+128 n^2+208 n-96\Big) ~,\\ \nonumber
R_1 &= & -2 (2 n+3) \Big(4 n^2 x^3+4 n x^3+x^3-4 n^2 x^2+8 n x^2-x^2\\ \nonumber
& & +16 n^2 x+8 n x+12 x+16 n-4\Big) ~.
\end{eqnarray}

From (\ref{star1}) and (\ref{Rs}),
we find that $Q_1 R,  R_1 Q,$ and $ R_1 Q_0- Q_1 R_0 $ in
\begin{equation}
\label{deform}
Q_1 R \frac{d^2 }{dx^2} H_0 - R_1 Q \frac{d }{dx}H_0  + (R_1 Q_0- Q_1 R_0 ) H_0=0  
\end{equation}
have GCD
$$
(2 n+3) (x-2) (x+2) (2 n x+x+2) (2 n x+3 x+2) ~.
$$
Dividing through (\ref{deform}) by this and defining $S_2, S_1, S_0$ as the resulting quotients, 
we obtain the second order differential equation satisfied by $H_0$. We record this in the following theorem.
\begin{theorem}
\label{detheorem}
Suppose the polynomials $a_k (x)$ and the $(n+1) \times (n+1)$ Hankel
determinant $H_0 = H_0 (n, x)$  are as defined in (\ref{aks}) and (\ref{Hn}).
Then
\begin{equation}
\label{de}
 S_2 \frac{d^2 }{dx^2} H_0 + S_1 \frac{d }{dx} H_0 + S_0 H_0 =0 ~,
\end{equation}
where
\begin{eqnarray*}
S_2 &= &  (x-2) (x+2) (4 n^2 x^2+4 n x^2+x^2+8 n x+4) ~,\\
S_1 &=&  2 (4 n^2 x^3+4 n x^3+x^3-4 n^2 x^2+8 n x^2-x^2+16 n^2 x+8 n x+12 x+16 n-4) ~,\\
S_0 &=&  -n (n+1) (4 n^2 x^2+4 n x^2+x^2+8 n x-8 x+36) ~.
\end{eqnarray*}

\end{theorem}

\section{Evaluation at special points}
\label{special}

At this point we have enough information to evaluate $H_0(x)$ at special points $x$ without making use of 
the differential equation (\ref{de}) itself.

Using the notation that incorporates the sizes of the matrices involved,
we recall the  following general result on Hankel determinants proved in \cite{ERR07}:
\begin{proposition}
\begin{equation}
\label{pfaff}
H_0 (n-1, x) H_0 (n+1, x) = H_0 (n, x) H_2 (n, x) + H_0(n, x) H_{1^2}  (n, x) - H_1(n, x)^2 ~.
\end{equation}
\end{proposition}

\subsection{Specialization at $x=2$}
At $x=2$, the derivative expression in (\ref{derivativeH0}) gives
$$
-2 (n+1 ) (3 + 6 n + 2 n^2)H_0 + (1 + 4 n + 2 n^2)H_1 = 0 ~.
$$
%The other derivative  (\ref{derivativeH1}) gives the same thing.
From equations (\ref{expansionH2}) and (\ref{expansionH11})
at $ x= 2$,
\begin{eqnarray*}
 (n+2 ) H_2 &=& 
- (n+1) (n+4 ) (2n+5 ) H_0 + (13 + 11 n + 2 n^2)H_1 ~,\\
( n+ 1 ) H_{1^2} &=& 
-n (n+1) (2n+5 ) H_0 + n (2n+3 ) H_1 ~.
\end{eqnarray*}

Therefore at $x=2$ we can write (\ref{pfaff}) as
$$
H_0 (n-1, 2) H_0 (n+1, 2) =  \frac{(2 n^2 - 1) (7 + 8 n + 2 n^2)}{(1 + 4 n + 2 n^2)^2} H_0(n, 2)^2
 ~.
$$
This is a recursion in $H_0 (n,2)/H_0(n-1,2)$ with $H_0(0,2)=1$, $H_0(1,2)=-7$. Solving, we find
$$
H_0 (n, 2) = (-1)^n (2 n^2+4 n+1) ~.
$$
At $x=2$, 
the entries of the determinant in (\ref{x2}) specialize to
\begin{equation}
\label{ak2exp}
a_k(2) = 4^{k + 1} - {2k+3 \choose k+1}.
\end{equation}
The evaluation of the corresponding Hankel determinant is as follows:
\begin{corollary}
Suppose $ a_k(x)$ is as defined in (\ref{aks}). Then
\begin{equation}
\label{x2}
H_0 (n, 2) = 
\det \left[ a_{i+j} (2) \right]_{0 \leq i,j \leq n} = (-1)^n (2 n^2+4 n+1) ~.
\end{equation}
\end{corollary}

\subsection{Specialization at $x=-2$}
At $x=-2$ 
the expression for the derivative in 
(\ref{derivativeH0}) gives
$$
-2 (n+1 ) (3 - 4 n + 6 n^2 + 4 n^3)H_0 + (2n+3 ) (1 + 2 n^2)H_1 = 0 ~.
$$
Again
from equations (\ref{expansionH2}) and (\ref{expansionH11})
we obtain at $ x= -2$,
\begin{eqnarray*}
(n+1) H_2 &=& 
-(n-1) (n+1 ) (2n+5 ) H_0 + (n+2 ) (2n+1 ) H_1 ~,\\
 n H_{1^2} &=& 
(2 - 4 n - 5 n^2 - 2 n^3) H_0 + (n+1) (2n-1 ) H_1 ~.
\end{eqnarray*}

Therefore we can use (\ref{pfaff}) at $x=-2$ and write
$$
H_0 (n-1, -2) H_0 (n+1, -2) = 
\frac{(2n+1 ) (2n+5 ) (3 - 4 n + 2 n^2) (3 + 4 n + 2 n^2)}{(2n+3 )^2 (1 + 2 n^2)^2}  H_0(n, -2)^2
 ~.
$$
This is a recursion in $H_0 (n,-2)/H_0(n-1,-2)$ with
$H_0 (0,-2) = 1$, $H_0(1,-2)=5 $, which can be solved to give the simple product evaluation
\begin{equation}
\label{xm2}
H_0 (n, -2) = \frac{1}{3} (2n+3 ) (1 + 2 n^2) ~.
\end{equation}
Therefore
\begin{corollary}
Suppose $ a_k(x)$ is as defined in (\ref{aks}). Then
\begin{equation}
\label{xm2det}
\det \left[a_{i+j} (-2)  \right]_{0 \leq i,j \leq n} 
= \frac{1}{3} (2n+3 ) (1 + 2 n^2) ~.
\end{equation}
\end{corollary}
The entries in (\ref{xm2det}) do not seem to have as simple an expression as the $ a_k(2)$ given in
(\ref{ak2exp}),
although
from the alternate expression for the generating function of the $a_k$, we get the
generating function of these numbers as
$$
\frac{1}{1-y-4y^2-y t} = \frac{2}{1 + {\sqrt{1 - 4y}} - 2y( 1 + 4y ) }
 ~,
$$
where $t$ is the generating function of the Catalan numbers, as in the proof of Lemma \ref{FIlemma} in
Appendix II.

\section{The differential equation solution}
\label{DEsolution}

Natural candidates for the 
expansion of the power series solution to 
the differential equation (\ref{de}) 
are around $x=2$ and $ x=-2$.

\subsection{Solution at $x=2$}

Putting
$$
H_0(x) = \sum_{k=0}^{\infty} b_k (x-2)^k ~,
$$
we find that the $ b_k$ satisfy
\begin{eqnarray*}
16 k (2 k+1) (2 n^2+4 n+1) b_k & = & 8 \Big(2 n^4+6 n^3-10 k^2 n^2+18 k n^2-n^2-16 k^2 n+26 k n\\
                              & & -7 n-3 k^2+4 k-1 \Big) b_{k-1} +2 \Big(8 n^4+20 n^3-16 k^2 n^2\\
                              & & +60 k n^2-46 n^2-20 k^2 n +68 k n-58 n -4 k^2+15 k-14\Big) b_{k-2} \\
& & + (n+3-k) (k+n-2) (2 n+1)^2 b_{k-3}
\end{eqnarray*}
for $k \geq 2$ 
with $b_k = 0$ for $ k<0$. From (\ref{de}), we get 
\begin{equation}
\label{rel1}
b_1 = \frac{ n (n+1) (2 n^2+4 n+3)}{6 (2 n^2+4 n+1)} b_0 ~,
\end{equation}
and therefore each $b_k$ is a multiple of $b_0$.
It can then be proved by induction that 
$$
b_k = \frac{2 n^2 +4 n+2k^2+1}{(2n^2+4n+1)(2k+1)} {n+k \choose 2k} b_0 ~.
$$
Since $ b_0 = H_0 (2)$,
$$
H_0(x) = \frac{H_0(2)}{2n^2+4n+1} \sum_{k=0}^n \frac{2 n^2 +4 n+2k^2+1}{2k+1} {n+k \choose 2 k} (x-2)^k
~.
$$
The determinants at $ x=2$ have the  simple evaluation we already found in 
(\ref{x2}), so that
\begin{equation}
\label{Hnatm2}
H_0(x) = (-1)^n \sum_{k=0}^n \frac{2 n^2 +4 n+2k^2+1}{2k+1} {n+k \choose 2k} (x-2)^k ~.
\end{equation}
The coefficients in (\ref{Hnatm2}) can be 
rewritten as binomial coefficients to obtain the expansion given in (\ref{thm11}) 
of Theorem \ref{thm1}  at $ x= 2$. Note that the alternate notation 
$H_0(x)$  in (\ref{Hnatm2}) (subscript indicating the zero
partition), is 
the $(n+1) \times (n+1)$ determinant 
denoted by $H_n(x)$ in Theorem  \ref{thm1}.

Using the expansion at $ x =2$, we can immediately write down the generating function of
the $ H_n(x)$. We omit the proof of the following result.

\begin{theorem}
\label{thmgf}
Suppose $ a_k(x)$ is as defined in (\ref{aks}). Then
\begin{equation}
\label{gf}
\sum_{n=0}^{\infty} H_n(x) t^n = 
\frac{1 -t +t^2 -t^3 -x t -3 x t^2}{\left(1+ x t + t^2 \right)^2} ~.
\end{equation}
\end{theorem}

\subsection{Solution at $x=-2$}

For the solution around $ x = -2$, put
$$
H_0(x) = \sum_{k=0}^{\infty} d_k (x+2)^k ~.
$$
We find that the $ d_k$ satisfy
\begin{eqnarray*}
16 k (2 k+3) (2 n^2+1)  d_k &=& -8 \Big(2 n^4+2 n^3-10 k^2 n^2+10 k n^2+7 n^2-4 k^2 n  +6 k n\\
                           & & +5 n-3 k^2+2 k+1\Big)  d_{k-1} + 2 \Big(8 n^4+12 n^3-16 k^2 n^2+52 k n^2\\
& & -30 n^2-12 k^2 n+44 k n-34 n-4 k^2+13 k-10\Big)  d_{k-2} \\
& & + (k-n-3) (k+n-2) (2 n+1)^2 d_{k-3}
\end{eqnarray*}
for $k \geq 2$ 
with $d_k = 0$ for $ k<0$. From (\ref{de}), we get 
\begin{equation}
\label{rel2}
d_1 = -\frac{n (1 + n) (7 + 2 n^2)}{10 (1 + 2 n^2)} d_0 ~,
\end{equation}
and therefore each $d_k$ is a multiple of $d_0$.
It can be proved by induction that 
$$
d_k = (-1)^k \frac{ 3(2n^2 + 2k^2 + 4k + 1)}{(1 + 2n^2)(2k + 1) (2k + 3)} {n + k \choose 2k} d_0 ~.
$$
Since $ d_0 = H_0 (-2)$,
$$
H_0(x) = \frac{H_0(-2)}{1+2n^2} \sum_{k=0}^n (-1)^k 
\frac{ 3(2n^2 + 2k^2 + 4k + 1)}{(2k + 1)(2k + 3)} {n + k \choose 2 k} (x+2)^k ~.
$$

Using the evaluation of the determinants
at $ x=-2$ from
(\ref{xm2}) we obtain
\begin{equation}
\label{Hnat2}
H_0(x) = (2n+3) \sum_{k=0}^n (-1)^k 
\frac{2n^2 + 2k^2 + 4k + 1}{(2k + 1)(2k + 3)} {n + k \choose 2 k} (x+2)^k ~.
\end{equation}
The coefficients in (\ref{Hnat2}) can be rewritten in the form (\ref{thm11}) 
of Theorem \ref{thm1}. Again, note that 
$H_0(x)$  in (\ref{Hnat2}) is
the $(n+1) \times (n+1)$ determinant $H_n(x)$ in Theorem  \ref{thm1}.

Evaluating 
(\ref{thm11}) and (\ref{thm12})
at $x=0$ we obtain the expressions
\begin{eqnarray}
\label{at0}
&& \hspace*{-10mm} \det \left[  2 (i +j) +2 \choose i+j \right]_{0 \leq i,j \leq n}   = 
(-1)^n \sum_{k=0}^n \left[  (2n+3) {n+k \choose 2k+1} + (2k+1) {n+k+1 \choose 2k+1} \right] (-2)^k\\
%(-1)^n \sum_{k=0}^n \frac{2 n^2 +4 n+2k^2+1}{2k+1} {n+k \choose 2k} (-2)^k\\
 && \hspace*{3mm} =
 \sum_{k=0}^n \left[  ( n+k+1) {n+k \choose 2k} + (2n+4k+1) {n+k \choose 2k+1}+ 8(k+1)  {n+k +1 \choose 2k+3} \right] (-2)^k~.
 %(2n+3) \sum_{k=0}^n \frac{2n^2 + 2k^2 + 4k + 1}{(2k + 1)(2k + 3)} {n + k \choose 2k} (-2)^k ~.
\end{eqnarray}
which are alternate ways of writing the known evaluation of this determinant from
(\ref{m1}).

As another corollary of
Theorem \ref{thm1},
we have the following Hankel determinant evaluation at $ x=1$,
which depends on the residue class of $n$ modulo 3:

\begin{corollary}
\label{additional}
\dsp
\begin{equation}
\label{stardet1}
\det \left[ { 2 (i+j) +3 \choose i+j} \right]_{0 \leq i,j \leq n} =
  \left\{ 
\begin{array}{ll}
\frac{1}{3} (2n+3)  & \mbox{if } n ~\equiv~ 0 \; (\mbox{\rm mod}  \; 3)~, \\
 - \frac{4}{3} (n+2) & \mbox{if } n ~\equiv~ 1 \;(\mbox{\rm mod} \;3)~, \\
\frac{1}{3} (2n+5) & \mbox{if } n ~\equiv~2 \;(\mbox{\rm mod} \; 3)~.
\end{array}
\right.
\end{equation}
\ssp
\end{corollary}
\begin{proof}
Since
\begin{equation}
\label{stardet2}
a_k(1) = { 2 k +3 \choose k}
\end{equation}
the determinant is simply $H_0(n,1)$.
We use 
the expression for the determinant
(\ref{thm11}) of  Theorem \ref{thm1} evaluated at $x=1$.
Putting $n =3m$, $n =3m+1$ and $n =3m+2$ for the three residue classes modulo 3, the Corollary is a
consequence of the resulting binomial identities
\dsp
\begin{eqnarray}
2m+1 &= &
 (-1)^m \sum_{k=0}^{3m} \left[  (6m+3) {3m+k \choose 2k+1} + (2k+1) {3m+k+1 \choose 2k+1} \right] (-1)^k \\
%(-1)^m \sum_{k=0}^{3m}  \frac{1 + 2 k^2 + 12 m + 18 m^2}{2k+1} {3m + k \choose 2k}(-1)^k\\
4(m + 1) &= & 
 (-1)^m \sum_{k=0}^{3m+1} \left[  (6m+5) {3m+k+1 \choose 2k+1} + (2k+1) {3m+k+2 \choose 2k+1} \right](-1)^k
\\
%(-1)^m \sum_{k=0}^{3m+1}  \frac{7 + 2 k^2 + 24 m + 18 m^2}{2k+1} {3m +1+ k \choose 2k}(-1)^k\\
2m+3  &=  & 
 (-1)^m \sum_{k=0}^{3m+2} \left[  (6m+7) {3m+k +2\choose 2k+1} + (2k+1) {3m+k+3 \choose 2k+1} \right] (-1)^k 
%(-1)^m \sum_{k=0}^{3m+2}  \frac{17 + 2 k^2 + 36 m + 18 m^2}{2k+1} {3m +2+ k \choose 2k}(-1)^k
\end{eqnarray}
\ssp
which can be proved by making use of the generating function given in Theorem \ref{thmgf} at $ x=1$.
\end{proof}

\subsection{Solution at $x=0$}

The power series solution to 
(\ref{de}) around $x=0$ is more difficult to derive directly.
For $a_k$ and the $H_n(x)$ as defined in 
(\ref{aks}) and (\ref{Hn}) this expansion is given by
\begin{equation}
\label{solutionat0}
H_n(x) = \sum_{k=0}^n (-1)^{ n(n-1)/2 + k(k-1)/2 +k n} \left( 2 k +(-1)^{n-k} \right) 
\frac{(n - \lfloor \frac{n-k+1}{2} \rfloor )!}{\lfloor \frac{n-k}{2} \rfloor ! k! } x^k ~.
\end{equation}
We are grateful to the anonymous referee for pointing out the above explicit form of the 
determinant around $ x=0$.
This
expansion is an immediate consequence of the
generating function for the determinants at arbitrary $x$ that 
we have provided in (\ref{gf}). 

Following the route of the proofs of the cases $ x=2 $ and $ x=-2$, one would
put
$$
H_n(x) = \sum_{k=0}^{\infty} e_k x^k ~,
$$
and show that the $ e_k$ satisfy the recursion
\begin{eqnarray*}
16 k (k-1) e_k & = & -8 (k-1) (4 k n-12 n+1) e_{k-1}\\
& & -4 \left(4 n^2 k^2+4 n k^2-28 n^2 k-24 n k-6 k+49 n^2+41 n+12\right) e_{k-2}\\
& & -2 \left(4 n^3+4 k n^2-12 n^2-4 k^2 n+20 k n-28 n+k-3\right)e_{k-3}\\
& & + (k-n-4) (k+n-3) (2 n+1)^2 e_{k-4}
\end{eqnarray*}
for $k \geq 2$ with $e_k = 0$ for $ k<0$. In this case each $e_k$ is a function of $e_0$ and $e_1$.
We know $e_0$ explicitly by (\ref{m1}).
However in this case a relationship similar to (\ref{rel1}) 
and (\ref{rel2}) of the $x=2$ and $x=-2$
cases does not drop out of the differential equation
to give a similar relation between
$e_0$ and $e_1$. This is because the special values of $ x$ that kills off 
the second derivative term in 
Theorem \ref{detheorem} are $ x= \pm 2$.
An alternate approach is to show directly that 
the coefficient of $ x^k$ in (\ref{solutionat0})
%\begin{equation}
%\label{eks}
%(-1)^{ n(n-1)/2 + k(k-1)/2 +k n} \left( 2 k +(-1)^{n-k} \right) 
%\frac{(n - \lfloor \frac{n-k+1}{2} \rfloor )!}{\lfloor \frac{n-k}{2} \rfloor ! k! } 
%\end{equation}
satisfies the recurrence for the $e_k$, but again this would
fall back on the already proved expansions of 
Theorem \ref{thm1} for the value of the derivative at $x=0$.

\section{Zeros of $H_0(n,x)$}
\label{zeros}

The determinants $H_0(n,x)$ of 
Theorem \ref{thm1}
are not orthogonal polynomials. But they 
satisfy a recurrence relation with polynomial coefficients involving three consecutive terms of the sequence 
as follows:

\begin{corollary}
\begin{eqnarray} \nonumber
& & (2 + (2 n +3) x)^2 H_0(n + 2, x) + x (4 + 4 (2n + 3) x + (2n + 3) (2n + 5) x^2) H_0(n + 1, x)  \\
&& \hspace*{2cm} + (2 +  (2 n+5) x)^2 H_0(n, x) = 0 ~.
\label{recursion}
\end{eqnarray}
\end{corollary}
\begin{proof}
The recurrence relation can be verified by making use of the explicit form of 
$H_0(n,x)$ from Theorem \ref{thm1}.
\end{proof}

Table \ref{zerostable} gives a list of the zeros of $H_0(1,x)$ through $H_0(7,x)$.
The zeros are real and interlacing. 
It is possible that
the polynomials  $ H_0(n, x)$ can be obtained from an orthogonal family by a suitable transformation.

\dsp
\begin{table}[htb]
\begin{center}
$-$0.333\\
$-$0.358 ~~ 0.558\\
$-$0.601 ~~ $-$0.194 ~~1.224\\
$-$1.083 ~~ $-$0.207 ~~0.324 ~~1.522\\
$-$1.367 ~~ $-$0.351 ~~ $-$0.137 ~~0.815 ~~1.678\\
$-$1.540 ~~ $-$0.746 ~~ $-$0.146 ~~0.229 ~~1.127 ~~1.768\\
$-$1.651 ~~ $-$1.028 ~~ $-$0.246 ~~ $-$0.107 ~~0.608 ~~1.333 ~~1.825
\end{center}
\caption{Zeros of the Hankel determinants  $H_0(1,x)$ through $H_0(7,x)$ of Theorem \ref{thm1}.}
\label{zerostable}
\end{table}
\ssp

A sequence of polynomials $ \{ P_n ( x) \}_{n \geq 0}$ with $ \deg P_n = n $ is called a 
{\em Sturm sequence} on an
open interval $( a, b )$ if $P_n$ has exactly $n$ simple real zeros in $(a,b)$, 
and for every $ n \geq 1$, zeros of $P_n(x)$ and $P_{n+1}(x)$ strictly interlace.

\begin{theorem}
\label{thm3}
Suppose $a_k$ and the $H_0 (n,x)$ are as defined in 
(\ref{aks}) and (\ref{Hn}). 
Then $\{ H_0(n,x)\}_{n \geq 0} $ is a Sturm sequence on $ (-2,2)$.
%Then $H_0(n,x)$ has $n$ simple real roots in the interval $ (-2,2)$.
%Moreover the zeros of 
%$H_0(n,x)$ and the zeros of 
%$H_0(n+1,x)$ are interlaced.
\end{theorem}
\begin{proof}

Consider the two expansions 
of $H_0 (n,x)$ in (\ref{Hnatm2}) and (\ref{Hnat2}). 
The first one of these implies that $(-1)^n H_0 (n,x) > 0 $ for $ x \geq 2$, and the 
second one implies that $H_0 (n,x ) >0 $ for $ x \leq -2$. Therefore the zeros of  $H_0 (n,x ) $ 
are contained in $(-2,2)$.

We next prove that 
like orthogonal polynomials,
$H_0(n,x)$ has $n$ distinct real zeros 
and the zeros of $H_0(n,x)$ lie strictly between the zeros of $H_0(n+1,x)$.
This interlacing property is a consequence of the form of the recursion (\ref{recursion})
\begin{equation}
\label{squares}
\alpha^2 H_0(n + 2, x) + x \beta H_0(n + 1, x) 
+ \gamma^2 H_0(n, x) = 0
\end{equation}
where $ \beta > 0$ for every $ x$ and $ n$. We use induction on $n$.
For any two consecutive zeros $ r_1, r_2$ of $ H_0(n + 1, x) $ the induction hypothesis implies that 
$ H_0(n, r_1) $ and
$ H_0(n, r_2) $
have opposite signs. Therefore from the recursion, $ H_0(n+2, r_1) $ and
$ H_0(n+2, r_2) $ also have opposite signs and so $H_0(n + 2, x) $
has at least one zero in the interval $ (r_1, r_2)$. This accounts for $ \geq n$ zeros of
$H_0(n + 2, x) $.  Let $ \delta_2 < 2 $ be the largest zero of $ H_0(n + 1, x) $. By the induction hypothesis, 
 $ H_0(n , x) $ has no zeros on $ [\delta_2 , \infty )$. Therefore its sign at $x=\delta_2 $ is the 
same as its sign at $ x=2$, which is $(-1)^n $. But the sign of $H_0(n + 2, x) $ is also 
 $(-1)^n $ at  $ x=2$, but opposite of the sign of  $ H_0(n , x) $ at $ x=\delta_2$ by (\ref{squares}). This forces 
 $H_0(n + 2, x) $ to change sign and 
have a zero in $ (\delta_2, 2)$. By a counting argument, $H_0(n + 2, x) $ has to 
have another zero in $(-2,\delta_1)$ where $\delta_1$ is the smallest zero of $ H_0(n + 1, x)$.
\end{proof}

\section{Discussion, patterns and conjectures}
\label{discussion}
We introduced a class of multilinear  operators $ \gamma$ acting on tuples of matrices to take the 
place of the trace method of our earlier calculations. This approach to evaluate Hankel determinants is 
easier to work with: the  $\gamma$-operators are  easier to differentiate, and they do not produce the 
extraneous nonlinear terms. In the $(2,2)$-case that we have covered in detail, we have also obtained 
numerical evaluations at special points as a byproduct. Furthermore we saw that the resulting 
polynomials have intriguing properties.

Even though the application of the $\gamma$-operator reduces the calculations involved in almost product 
evaluations of Hankel determinants considerably, there are still stumbling blocks in
the general $(2,r)$-case, and other cases that differ little from this.
We consider a few of these determinants and conjecture closed forms for the evaluations.

Corollary \ref{additional} is just one example of a strange pattern that holds for Hankel determinants 
where the entries are the polynomials $a_k^{(2 , r )} (x) $ defined in (\ref{ab}). Taking $x=0$, let
$$
a_k = { 2 k +r \choose k} ~,
$$
parametrized by $ r \geq 0$. For notational simplicity, define 
$$
   F(n,r) = \mbox{det}\left[a_{i+j} \right]_{0 \leq i,j \leq n}.
$$
Then the evaluation (\ref{stardet1}) in Corollary \ref{additional} can be written as 
\begin{eqnarray*}
F(3m ,3 ) &=& 2m+1   \\
 F(3m+1, 3) &=& - 4 (m+1)  \\
F(3m+2, 3 ) &=& 2m+3 ~.
\end{eqnarray*}
As an example, consider the following evaluations 
for the case $r = 7$:
\begin{eqnarray*}
   F(7m,7)   & = & (2m + 1)^3 \\
   F(7m+1,7) & = & (m+1)(2m+1)^2(9604m^3+9604m^2-1323m-2340)/90 \\
   F(7m+2,7) & = & -(m+1)^2(2m+1)(19208m^3+67228m^2+70854m+23445)/45 \\
   F(7m+3,7) & = & 64(m+1)^3 \\
   F(7m+4,7) & = & (m+1)^2(2m+3)(19208m^3+48020m^2+32438m+3015m)/45 \\
   F(7m+5,7) & = & -(m+1)(2m+3)^2(9604m^3+48020m^2+75509m+38110)/90 \\
   F(7m+6,7) & = & (2m+3)^3 ~.
\end{eqnarray*}
These evaluations have been verified for a significant range of $m$.
This unusual set of formulas is typical of a complex pattern of evaluations of $F(n,r)$
that continues with several unexpected dependencies on the value of $n$ modulo $r$ and on $r$ modulo 4.  
For example, if $r$ is odd then there is strong experimental  evidence that
$$
   F(r m, r) = F(r m - 1, r) =  (2 m + 1)^{(r-1)/2} ~.
$$
When we consider even $r$ there is another twist to take into account.  
Experimental evidence tells us that 
\dsp
$$
F(r m, r) = F(r m - 1, r)  =
  \left\{ 
\begin{array}{ll}
1  & \mbox{if } r ~\equiv~ 0 \; (\mbox{\rm mod}  \; 4)~, \\
 (-1)^m & \mbox{if } r ~\equiv~2 \;(\mbox{\rm mod} \; 4)~.
\end{array}
\right.
$$
\ssp
Another interesting pattern we observe is the following for odd $r$:
$$
   F(r m + (r-1)/2, r) = 2^{r-1}(m+1)^{(r-1)/2} ~.
$$
For even $r$ there is also a simple pattern of this type:
\dsp
$$
 F(r  m + r/2, r) =
  \left\{ 
\begin{array}{ll}
 (-1)^{r/4+1}(2r(m+1))^{r/2-1} & \mbox{if } r ~\equiv~ 0 \; (\mbox{\rm mod}  \; 4)~, \\
& \\
 (-1)^{(r+2)/4+m}(2r(m+1))^{r/2-1} & \mbox{if } r ~\equiv~2 \;(\mbox{\rm mod} \; 4)~.
\end{array}
\right.
$$
In addition to these nice evaluations there are many that are not so
simple.  For example the $F(rm+1,r)$ becomes more and more complex as $r$ increases.  For $r=5$
$$
   F(5m+1,5)=-(m+1)(2m+1)(50m+39)/3.
$$
For $r=7$ the evaluation contains a cubic factor:
$$
   F(7m+1,7)=(m+1)(2m+1)^2(9604m^3+9604m^2-1323m-2340)/90
$$
and when $r=9$ the evaluation contains a quartic factor:
$$
   F(9m+1,9)=-(m+1)(2m+1)^3(3m+2)(52488m^4+69984m^3+22518m^2+1674m+1505)/70 ~.
$$

We suspect that this  irreducible factor keeps gaining a degree when $r$ is increased by 2.

These conjectures appear to be difficult to prove in their full generality using either
the methods described in Krattenthaler \cite{K99,K05} or with the methods of the present paper. 
For any fixed $r$, the methods of this paper might apply but it is 
hard to see how to approach the problem when $r$ is left as a parameter.

Further experimental evidence suggests that the determinants 
$$
  \det \left[\sum_{k=0}^{i+j}
        {2i + 2j + r - 2k \choose i+j-k}x^k\right]_{0 \leq i,j \leq n}
$$
satisfy second order differential equations.  However as $r$ gets
larger the differential equations and the first and second identities of our method
become increasingly complex.  We mention that there are also difficulties in evaluating the 
family of determinants 
\begin{equation}
\label{fourth}
  \det \left[\sum_{k=0}^{i+j}
        {2i + 2j + r - k \choose i+j-k}x^k\right]_{0 \leq i,j \leq n} ~.
\end{equation}
For this family, the order of the differential equation for the determinant seems 
to increase with $r$.
When $r=4$, for example, experiments suggest that (\ref{fourth}) satisfies 
a fourth order differential equation.  

\subsection*{Acknowledgments}

We would like to 
thank the anonymous referees for several constructive remarks 
including the explicit form (\ref{solutionat0}) of the determinant at
$x=0$.

\newpage

\bibliographystyle{plain}

\begin{thebibliography}{1}
\bibitem{Z08} D. Zeilberger.
{\tt   MultiZeilberger} Maple package,
or after extracting coefficients
{\tt SumTools[Hypergeometric]} with command {\tt ZeilbergerRecurrence}, or the package
{\tt EKHAD} can be used, all available from Doron Zeilberger's website.
\end{thebibliography}

\newpage
\section{Appendix I}

The results in this Appendix apply to  general Hankel matrices.
We let $ \chi (S)$ denote the indicator of the statement $S$: $\chi(S) = 1$ if $S$ is true and 
$\chi (S)=0$ if $S$ is false.

\subsection{Properties of the $\gamma$-operator}

\noindent
{\bf Proposition \ref{greplace}}

{\it 
For $ m \leq n+1$,
\begin{equation}
\gamma_A ( X_1, \ldots , X_m ) = \sum_{S, \sigma} \det ( A_{S, \sigma})
\end{equation}
where the summation is over all subsets
$S$ of
$\{ 0, 1, \ldots, n\}$ with $ |S| = m$ and all permutation $ \sigma$ of
$\{ 1, 2, \ldots, m \}$.
}
\begin{proof}
Expand 
$$
\det ( A + t_1 X_1 + t_2 X_2 + \cdots + t_m X_m ) 
$$
by columns (or rows)
using the linearity of the determinant to obtain
\begin{equation}
\label{dexp}
\det ( A + t_1 X_1 + t_2 X_2 + \cdots + t_m X_m )  =
t_1 t_2 \cdots t_m \sum_{S, \sigma} \det ( A_{S, \sigma} ) 
\end{equation}
where $A_{S, \sigma}$ is as defined in 
Definition \ref{Ass}. 
The proof follows by 
applying 
$\partial_{t_1} \partial_{t_2} \cdots \partial_{t_m}$ and putting $  t_1 = \cdots  = t_m =0$. 
\end{proof}

\noindent
{\bf Proposition \ref{gder}}

{\it
For $m \leq n$,
$$
\frac{d}{dx} \gamma_A ( X_1, \ldots, X_m) =
\gamma_A ( \frac{d}{dx} A , X_1, \ldots , X_m)
+ \sum_{j=1}^m \gamma_{A} ( X_1, \ldots, X_{j-1}, \frac{d}{dx} X_j, X_{j+1}, \ldots , X_m)
$$
}
\begin{proof}
By Proposition \ref{greplace} and the expression in 
(\ref{dexp}), 
\begin{eqnarray*}
\frac{d}{dx} \gamma_A ( X_1, \ldots, X_m) &=&
\sum_{S, \sigma} \frac{d}{dx} \det ( A_{S, \sigma} )  \\
&= & \sum_{S, \sigma} \det ( A_{S, \sigma} ) \mbox{Tr} ( A_{S, \sigma}^{-1} \frac{d}{dx}  A_{S, \sigma} )
\end{eqnarray*}
Let $ B = A_{S, \sigma}$.
By Cramer's rule, 
$$
\mbox{Tr} ( B^{-1} \frac{d}{dx}  B ) =\frac{1}{\det(B)} \sum_{j=0}^n \det (B_j) 
$$
where $B_j$ is obtained from $B$ by replacing the $j$-th column of $B$ by its derivative.
In terms of the matrix $A$, let 
$A_{S, \sigma, j } $ denote this matrix. 

Therefore
\begin{eqnarray*}
\frac{d}{dx} \gamma_A ( X_1, \ldots, X_m) 
&=& \sum_{S, \sigma} \sum_{j=0}^n \det ( A_{S, \sigma, j} )  \\
&=& \sum_{j=0}^n \sum_{S, \sigma}  \chi ( j \not\in S)  \det ( A_{S, \sigma, j} ) +
 \sum_{j=0}^n \sum_{S, \sigma}  \chi ( j \in S) \det ( A_{S, \sigma, j} ) \\
&= & \gamma_A ( \frac{d}{dx} A , X_1, \ldots , X_m)
+ \sum_{j=1}^m \gamma_{A} ( X_1, \ldots, X_{j-1}, \frac{d}{dx} X_j, X_{j+1}, \ldots , X_m)
\end{eqnarray*}
\end{proof}

\subsection{Expansion of the convolution matrices}

The expansion of the convolution matrices $ [ c_{i+j+k}]$ for $ k \geq -1$ are as follows:
\begin{proposition}
Suppose the convolution polynomial $c_n$ is as defined in Definition \ref{dcon}. Then
\begin{eqnarray}
\label{cex}
[ c_{i+j+k}]_{0 \leq i,j \leq n} & = & 
\sum_{p=0}^{n+k}
a_p [ a_{i+j+k-p} \chi( j \geq p-k) ]_{0\leq i,j \leq n} \\ \nonumber
& & \hspace*{1.5cm} + \sum_{p=0}^{n-1}
a_p [ a_{i+j+k-p} \chi( i > p ) ]_{0\leq i,j \leq n}
\end{eqnarray}
\end{proposition}
\begin{proof}
The $(i,j)$-th entry of the matrix on the right-hand side of (\ref{cex}) is
$$
\sum_{p=0}^{n+k}
a_p  a_{i+j+k-p} \chi( j \geq p-k) + \sum_{p=0}^{n-1} a_p  a_{i+j+k-p} \chi( i > p )
$$
The upper limit of the sums need not go past $i+j+k$.  In the second sum, replace
$p$ by $i+j+k-p$ and rearrange the indices. We get 
\begin{eqnarray*}
\sum_{p=0}^{i+j+k}
a_p  a_{i+j+k-p} \chi( j \geq p-k) +\sum_{p=0}^{i+j+k} a_p  a_{i+j+k-p} \chi( j <p-k ) &=&
\sum_{p=0}^{i+j+k}
a_p  a_{i+j+k-p}\\
&=& c_{i+j+k} 
\end{eqnarray*}
\end{proof}
Below are a few examples of the expansion of the convolution matrices obtained from (\ref{cex}).
For $ k=-1$,
\begin{eqnarray*}
[ c_{i+j-1}]_{0 \leq i,j \leq 2}  & =&   
a_0 \left[ 
\begin{array}{ccc}
0& a_0 & a_1  \\
0& a_1 & a_2  \\
0& a_2 & a_3  
\end{array} 
\right]
+
a_1 \left[ 
\begin{array}{ccc}
0& 0& a_0  \\
0& 0& a_1   \\
0& 0& a_2   
\end{array} 
\right] \\
&  &
\hspace*{1cm} + ~a_0 \left[ 
\begin{array}{ccc}
0 & 0 & 0 \\
a_0 & a_1 & a_2 \\
a_1 & a_2 & a_3 
\end{array} 
\right]
+
a_1 \left[ 
\begin{array}{ccc}
0& 0 & 0  \\
0& 0 & 0  \\
a_0& a_1 & a_2  
\end{array} 
\right]
\end{eqnarray*}

For $k=0$,
\begin{eqnarray*}
[ c_{i+j}]_{0 \leq i,j \leq 2}  & =&   
a_0 \left[ 
\begin{array}{ccc}
a_0 & a_1 & a_2 \\
a_1 & a_2 & a_3 \\
a_2 & a_3 & a_4 
\end{array} 
\right]
+
a_1 \left[ 
\begin{array}{ccc}
0& a_0 & a_1  \\
0& a_1 & a_2  \\
0& a_2 & a_3  
\end{array} 
\right]
+
a_2 \left[ 
\begin{array}{ccc}
0& 0& a_0   \\
0& 0& a_1   \\
0& 0& a_2   
\end{array} 
\right]\\
&  &
\hspace*{1cm} + ~a_0 \left[ 
\begin{array}{ccc}
0 & 0 & 0 \\
a_1 & a_2 & a_3 \\
a_2 & a_3 & a_4 
\end{array} 
\right]
+
a_1 \left[ 
\begin{array}{ccc}
0& 0 & 0  \\
0& 0 & 0  \\
a_1& a_2 & a_3  
\end{array} 
\right]
\end{eqnarray*}
For $k=1$, the expansion is as given in  (\ref{k1}).

\newpage
\section{Appendix II}

In this 
Appendix we give the proofs of the statements needed for the 
$(2,2)$-case.

The proofs of Lemma \ref{FIlemma},  Lemma \ref{SIlemma} are based on 
generating function manipulations, as given below.
The first identity for the $(2,2)$-case is:

\noindent
{\bf Lemma \ref{FIlemma}} 

\vspace*{-2mm}

%\begin{lemma}
%\label{thm12np2}
\begin{eqnarray} \nonumber
\label{thm12np2proof}
& &  (x-2) x (x+2) (3 x+2) \frac{d}{dx} a_n + \left( 3 x^3+18 x^2-20 x+24+4 n \left(x^2+4\right)\right) a_n
    \\ 
& &- \left( n (x-6) (x-2)+ 3 x^2-2 x+4\right)a_{n+1} -2 n (x-1) a_{n+2} \\ \nonumber
& &  -8  (x-1)^2 c_n +32  (x-1)^2 c_{n-1} =0
\end{eqnarray}
%\end{lemma}
\begin{proof}
From \cite{PS1925,ERR07}, we have
$$
f = f (x,y)  =  \sum_{k \geq 0} a_k(x) y ^k =
  \frac{t^3}{ (2 -t ) ( 1 - xy t)} ~.
$$
Here
$$
t= \sum_{k\geq 0 }\frac{(2k)!}{(k+1)! k!} y^k  = 1+ y +2y^2 +5y^3 + \cdots 
$$ 
satisfies 
\begin{equation}
\label{yt2}
y t^2 = t-1 ~.
\end{equation}
Using
$ \frac{d}{dy} t = {t^2}/( 1-2yt)$ in the computation of $ \frac{d}{dy}  f$
and using the resulting expressions for $ \frac{d}{dx}f $ and $f' = \frac{d}{dy} f$,
we make the substitutions
\begin{eqnarray*}
\frac{d}{dx} a_n & \rightarrow & \frac{d}{dx} f \\
a_n & \rightarrow & f \\
n a_n & \rightarrow & y f' \\
 a_{n+1} & \rightarrow & (f-1)/y \\
 n a_{n+1} & \rightarrow & y ((f-1)/y)' \\
 n a_{n+2} & \rightarrow & y ((f-1-(4+x)y)/y^2)' \\
c_n & \rightarrow & f^2 \\
c_{n-1} & \rightarrow & y f^2 
\end{eqnarray*}
in the left-hand side of (\ref{thm12np2proof}). The resulting 
expression
factors as 
\begin{eqnarray*}
& & \frac{(t-1-y t^2)}{(t-2)^2 y^2 (1-2 t y) (1- t x y)^2} \Big( 64 x^2 y^3 t^5-128 x y^3 t^5+64 y^3 t^5-16 x^2 y^2
t^5+ 32 x y^2 t^5-16 y^2 t^5  \\
& & +20 x^3 y^3 t^4-16 x^2 y^3 t^4+ 16 x y^3 t^4-x^3 y^2 t^4-4 x^2 y^2 t^4 -12 x y^2 t^4+32 y^2 t^4+6 x^2 y t^4\\
& & +2 x y t^4-8 y t^4-32 x^3 y^3 t^3+32 x^2 y^3 t^3-6 x^3 y^2 t^3-4 x^2 y^2 t^3+ 40 x y^2 t^3-80 y^2 t^3\\
& & -4 x t^3 -4 x^2 y t^3- 4 x y t^3+8 y t^3  +4 t^3 + 16 x^3 y^2 t^2+ 48 x^2 y^2 t^2-64 x y^2 t^2+4 x t^2\\
& & -16 x y t^2 +16 y t^2-4 t^2-32 x^2 y t+32 y t+16 x-16 \Big)
\end{eqnarray*}
and therefore vanishes by (\ref{yt2}).
\end{proof}

The second identity is:

\noindent
{\bf Lemma \ref{SIlemma}} 

\vspace*{-2mm}

%\begin{lemma}
%\label{thm22np2}
\begin{eqnarray} \nonumber
\label{thm22np2proof}
& & (n x + 3 x + 2) a_{n+2} - (n x (x + 6) + 3 x^2 + 16 x + 8) a_{n+1}
 + 2x (x + 2)(2n + 5 ) a_n  \\ 
& & \hspace*{2cm} +(x - 1) (x - 2) c_n -  4 (x - 1) (x - 2) c_{n-1}  =0
\end{eqnarray}
%\end{lemma}
\begin{proof}
Again passing to the generating functions, we find that the generating function of the 
left-hand side of 
(\ref{thm22np2proof}) factors as

% Checked in 2n+2_identities_Dec11_2006.nb
\begin{eqnarray*}
& & \frac{(t -1 -y t^2)}{(t-2)^2 y^2 (1-2 t y) (1-t x y)^2}
\Big( 24 x y^3 t^5 -8 x^2 y^3 t^5-16 y^3 t^5+2 x^2 y^2 t^5-6 x y^2 t^5+4y^2 t^5\\
& & +8 x^3 y^3 t^4   +16 x^2 y^3 t^4 -x^3 y^2 t^4-6 x^2 y^2 t^4-4 x y^2 t^4-8 y^2 t^4+x y t^4+2 y t^4-8 x^3 y^3 t^3\\
& & -16 x^2 y^3 t^3 -12 x^2 y^2 t^3 -16 x y^2 t^3+16 y^2 t^3-x t^3+2 x^2 y t^3+10 x y t^3+4 y t^3-2 t^3 +4 x^3 y^2 t^2 \\
& & +24 x^2 y^2 t^2  +32 x y^2 t^2+x t^2 -4 x y t^2-8 y t^2+2 t^2-8 x^2 y t-24 x y t-16 y t+4 x+8 \Big)
\end{eqnarray*}
which again vanishes by (\ref{yt2}).
\end{proof}

The third identity is:

\noindent
{\bf Lemma \ref{TIlemma}} 
{\it
\begin{equation}
\label{TIproof}
\sum_{j=0}^{n+2} w_{n,j}(x) a_{i+j}(x) =0
\end{equation}
for $ i = 0 , 1, \ldots , n $ where
\begin{eqnarray} \nonumber
 w_{n,j}(x) & = & (-1)^{n-j} \left\{ \frac{2(2n+5)}{2j+1} { n+j+2  \choose 2j }
+ \frac{(2n+3)(2n+5)}{2j+1} { n+j+2  \choose 2j }  x \right. \\
& & \hspace*{2cm} \left . + \frac{(2n+3)(2n+5)}{2j+3} {  n+j+2 \choose 2j+1} x^2 \right\}
\label{TIweightsproof}
\end{eqnarray}
}
We do not give the proof of the third identity 
Lemma \ref{TIlemma} but remark that once the weights are guessed, the proofs of the 
identities can be left to automatic 
binomial identity provers such as {\tt MultiZeilberger} supplied by Doron Zeilberger 
(in Maple \cite{Z08}), and {\tt MultiSum} by Wegschaider (in Mathematica  \cite{Kurt}). The main step 
in finding the coefficients 
is interpolation and a symbolic algebra system (Mathematica in our case).

The weights in general can be found from the relation
\begin{equation}
\label{guess}
w_{n,n+2} H_{21^k} + w_{n,n+1} H_{1^{k+1}} + w_{n,n-k} H_0= 0
\end{equation}
which holds for $ k=0, 1, \ldots , n$. This can be seen by computing the determinant of the matrix
obtained from
$ A = [ a_{i+j}]_{0 \leq i,j \leq n}$ by replacing column $n-k$ by column $n$, and column $n$ by the 
zero vector written as sum of column vectors as indicated by the third identity. In the present case
this is (\ref{TIproof}).
Expanding,
all but three determinants vanish, giving (\ref{guess}).

We use (\ref{guess}) to guess third identities in general. For instance with offset 2 (i.e. the vectors
involved in the third identity are $v_0$ through $v_{n+2}$), it is possible to 
first guess  $w_{n,n+2} , w_{n,n+1} , w_{n,n}$ by linear algebra, then use  (\ref{guess}) to solve for
$w_{n,n-k} $ and consequently find the candidate coefficients by interpolation.

\newpage
\section{Appendix III: Tables of $\gamma$-operator evaluations}

The tables given in this Appendix apply to  general Hankel matrices.

\sssp

\begin{table}[htb]
\begin{center}
\begin{tabular}{rl}
$\gamma_A ( [a_{i+j}]) ~= $ & $ (n+1) H_ 0 $ \\
$\gamma_A ( [a_{i+j+1}]) ~= $ & $  H_1 $ \\
$\gamma_A ( [a_{i+j+2}]) ~= $ & $  H_2-  H_{1^2} $ \\
$\gamma_A ( [a_{i+j+3}]) ~= $ & $  H_3 - H_{21} +  H_{1^3} $ \\
$\gamma_A ( [a_{i+j+4}]) ~= $ & $  H_4 - H_{31} + H_{21^2} -  H_{1^4} $ \\
$\gamma_A ( [a_{i+j+5}]) ~= $ & $  H_5 - H_{41} + H_{31^2} - H_{21^3} + H_{1^5} $ \\
$\gamma_A ( [(i+j)a_{i+j}]) ~= $ & $ n(n+1) H_0 $ \\
$\gamma_A ( [(i+j)a_{i+j+1}]) ~= $ & $ 2nH_1 $ \\
$\gamma_A ( [(i+j)a_{i+j+2}]) ~= $ & $ 2n H_2 -2(n-1)H_{1^2} $ \\
$\gamma_A ( [(i+j)a_{i+j+3}]) ~= $ & $ 2n H_3 -2(n-1)H_{21}+2(n-2)H_{1^3} $ \\
$\gamma_A ( [(i+j)a_{i+j+4}]) ~= $ & $ 2n H_4 -2(n-1) H_{31} +2(n-2) H_{21^2} - 2(n-3) H_{1^4}  $ \\
$\gamma_A ( [(i+j)a_{i+j+5}]) ~= $ & $ 2n H_5 - 2(n-1) H_{41}   + 2 (n-2) H_{31^2} -2 (n-3) H_{21^3} + 2
(n-4) H_{1^5}  $ \\
$\gamma_A ( [c_{i+j-1}]) ~= $ & $ 0$ \\
$\gamma_A ( [c_{i+j}]) ~= $ & $  (2 n+1) a_0 H_0$ \\
$\gamma_A ( [c_{i+j+1}]) ~= $ & $  2a_0 H_1 + 2 n a_1 H_0$ \\
$\gamma_A ( [c_{i+j+2}]) ~= $ & $ 2a_0 H_2  - 2a_0 H_{1^2} +  2a_1 H_1+(2 n-1) a_2 H_0 $ \\
$\gamma_A ( [c_{i+j+3}]) ~= $ & $ 2a_0 H_3  - 2a_0 H_{21}  + 2a_0 H_{1^3} +  2a_1 H_2  $ \\
 & $ - 2a_1 H_{1^2} +  2a_2 H_1+(2 n-2) a_3 H_0  $ \vspace*{-5mm}
\end{tabular}
\end{center}
\caption{$ \gamma_A (*)$ computations.}
\label{table1}
\end{table}
\begin{table}[bht]
\begin{center}
\begin{tabular}{rl}
$\gamma_A ([a_{i+j+1}], [a_{i+j}]) ~= $ & $ n H_1 $ \\
$\gamma_A ([a_{i+j+1}], [a_{i+j+1}]) ~= $ & $ 2 H_{1^2} $ \\
$\gamma_A ([a_{i+j+1}], [a_{i+j+2}]) ~= $ & $  H_{21}- 2 H_{1^3} $ \\
$\gamma_A ([a_{i+j+1}], [a_{i+j+3}]) ~= $ & $    H_{31} -H_{2^2} -H_{21^2}+2H_{1^4} $ \\
$\gamma_A ([a_{i+j+1}], [a_{i+j+4}]) ~= $ & $     H_{41} - H_{32} - H_{31^2} + H_{2^21} + H_{21^3} - 2
H_{1^5} $ \\
$\gamma_A ([a_{i+j+1}], [(i+j)a_{i+j}]) ~= $ & $ n(n-1) H_1 $ \\
$\gamma_A ([a_{i+j+1}], [(i+j)a_{i+j+1}]) ~= $ & $ 2(2n-1)H_{1^2} $ \\
$\gamma_A ([a_{i+j+1}], [(i+j)a_{i+j+2}]) ~= $ & $ 2n H_{21} -2(2n-3)H_{1^3} $ \\
$\gamma_A ([a_{i+j+1}], [(i+j)a_{i+j+3}]) ~= $ & $ 2n H_{31} -2(n-1)H_{2^2} $ \\
 &  \hspace*{1cm} $ -2(n-1)H_{21^2}+2(2n-5)H_{1^4} $ \\
$\gamma_A ([a_{i+j+1}], [(i+j)a_{i+j+4}]) ~= $ & $ 2n  H_{41} - 2(n - 1) H_{32} - 2(n - 1)H_{31^2} $\\
 &  \hspace*{1cm} $ + 2(n - 2) H_{2^21} + 2(n - 2) H_{21^3} - 2(2n - 7) H_{1^5} $ \\
$\gamma_A ([a_{i+j+1}], [c_{i+j-1}]) ~= $ & $ -2 n a_0 H_0 $ \\
$\gamma_A ([a_{i+j+1}], [c_{i+j}]) ~= $ & $  (2 n-1) a_0 H_1- (2n-1) a_1 H_0 $ \\
$\gamma_A ([a_{i+j+1}], [c_{i+j+1}]) ~= $ & $ 4 a_0 H_{1^2} + 2 (n-1) a_1  H_1 - 2 (n-1) a_2 H_0 $ \\
$\gamma_A ([a_{i+j+1}], [c_{i+j+2}]) ~= $ & $ 2 a_0 H_{21} -4 a_0  H_{1^3} + 4  a_1 H_{1^2} +(2n-3) a_2
H_1 -(2n-3) a_3 H_0  $ \vspace*{-5mm}
\end{tabular}
\end{center}
\caption{$ \gamma_A ([a_{i+j+1}], *)$ computations.}
\label{table2}
\end{table}
\begin{table}[bht]
\begin{center}
\begin{tabular}{rl}
$\gamma_A ( [a_{i+j+2}], [a_{i+j}]) ~= $ & $ n H_2 - n H_{1^2} $ \\
$\gamma_A ( [a_{i+j+2}], [a_{i+j+1}]) ~= $ & $   H_{21}- 2 H_{1^3} $ \\
$\gamma_A ( [a_{i+j+2}], [a_{i+j+2}]) ~= $ & $  2H_{2^2}-2H_{21^2}+  2H_{1^4} $ \\
$\gamma_A ( [a_{i+j+2}], [a_{i+j+3}]) ~= $ & $  H_{32} - H_{31^2} - H_{2^21} + 2 H_{21^3} - 2 H_{1^5} $
\\
$\gamma_A ( [a_{i+j+2}], [(i+j)a_{i+j}]) ~= $ & $ n(n-1) H_2 - (n^2-n+2)H_{1^2} $ \\
$\gamma_A ( [a_{i+j+2}], [(i+j)a_{i+j+1}]) ~= $ & $  2(n-1)H_{21}  -4(n-1)H_{1^3} $ \\
$\gamma_A ( [a_{i+j+2}], [(i+j)a_{i+j+2}]) ~= $ & $ 2(2n-1)H_{2^2}-2(2n-2)H_{21^2}+2(2n-4)H_{1^4} $ \\
$\gamma_A ( [a_{i+j+2}], [(i+j)a_{i+j+3}]) ~= $ & $ 2 n   H_{32} - 2n H_{31^2} - 2 (n - 2) H_{2^21} +
4(n - 2) H_{21^3} - 4(n - 3) H_{1^5} $ \\
$\gamma_A ( [a_{i+j+2}], [c_{i+j-1}]) ~= $ & $ -2a_0 H_1 -2(n-1) a_1 H_0 $ \\
$\gamma_A ( [a_{i+j+2}], [c_{i+j}]) ~= $ & $ (2n-1)a_0 H_2 -(2n-1) a_0 H_{1^2} -2 a_1 H_1-
(2n-3) a_2 H_0 $  \vspace*{-5mm}
\end{tabular}
\end{center}
\caption{$ \gamma_A ([a_{i+j+2}], *)$ computations.}
\label{table3}
\end{table}
\begin{table}[thb]
\begin{center}
\begin{tabular}{rl}
$\gamma_A ( [a_{i+j+1}], [a_{i+j+1}], [a_{i+j}]) ~= $ & $ 2(n-1)H_{1^2} $ \\
$\gamma_A ( [a_{i+j+1}], [a_{i+j+1}], [a_{i+j+1}]) ~= $ & $ 6 H_{1^3} $ \\
$\gamma_A ( [a_{i+j+1}], [a_{i+j+1}], [a_{i+j+2}]) ~= $ & $ 2 H_{21^2} - 6 H_{1^4 } $ \\
$\gamma_A ( [a_{i+j+1}], [a_{i+j+1}], [a_{i+j+3}]) ~= $ & $ 2  H_{31^2} - 2 H_{2^21} - 2 H_{21^3} + 6
H_{1^5} $ \\
$\gamma_A ( [a_{i+j+1}],[a_{i+j+1}], [(i+j)a_{i+j}]) ~= $ & $ 2(n-1)(n-2) H_{1^2} $ \\
$\gamma_A ( [a_{i+j+1}],[a_{i+j+1}], [(i+j)a_{i+j+1}] ) ~= $ & $ 12 (n-1) H_{1^3} $ \\
$\gamma_A ( [a_{i+j+1}], [a_{i+j+1}], [(i+j)a_{i+j+2}] ) ~= $ & $ 4 n H_{21^2} - 12 (n-2) H_{1^4} $ \\
$\gamma_A ( [a_{i+j+1}], [a_{i+j+1}], [(i+j)a_{i+j+3}] ) ~= $ & $ 4 n H_{31^2} - 4(n - 1) H_{2^21} - 4(n
- 1) H_{21^3} + 12(n - 3) H_{1^5} $ \\
$\gamma_A ( [a_{i+j+1}], [a_{i+j+1}], [c_{i+j-1}]) ~= $ & $ -4 (n-1) a_0 H_1 +4 (n-1) a_1 H_0 $ \\
$\gamma_A ( [a_{i+j+1}],[a_{i+j+1}], [c_{i+j}]) ~= $ & $2 a_0 (2n-3) H_{1^2} -2 (2n-3) a_1 H_1 +2(2n-3)
a_2 H_0 $ \vspace*{-5mm}
\end{tabular}
\end{center}
\caption{$ \gamma_A ( [a_{i+j+1}] , [a_{i+j+1}],  *)$ computations.}
\label{table4}
\end{table}

\end{document}